\documentclass[11pt]{article}

\usepackage{amssymb,amsmath,amsthm,enumitem,mathrsfs,mathabx,accents,xy}
\xyoption{all}
\usepackage{tikz}
\usetikzlibrary{cd}

\makeatletter
\renewenvironment{proof}[1][\proofname]{\par
  \pushQED{\qed}%
  \normalfont \topsep-5\p@\@plus6\p@\relax
  \trivlist
  \item[\hskip\labelsep
    #1\@addpunct{.}]\ignorespaces
}{%
  \popQED\endtrivlist\@endpefalse
}
\makeatother

\usepackage{titlesec}
\titlespacing\section{0pt}{-3pt plus 2pt minus 1pt}{-7pt plus 2pt minus 1pt}
\titlespacing\subsection{0pt}{-3pt plus 2pt minus 1pt}{-7pt plus 2pt minus 1pt}
\titlespacing\subsubsection{0pt}{10pt plus 4pt minus 2pt}{1pt plus 2pt minus 2pt}

\numberwithin{equation}{section}

\newtheoremstyle{reduced_space}{}{-0.5\baselineskip}{}{}{\bfseries}{}{.5em}{}
\theoremstyle{reduced_space}

\newtheorem{thm}{Theorem}[section]
\newtheorem{prp}[thm]{Proposition}
\newtheorem{lmm}[thm]{Lemma}
\newtheorem{crl}[thm]{Corollary}

\theoremstyle{definition}
\newtheorem{dfn}[thm]{Definition}

\theoremstyle{remark}
\newtheorem{rmk}[thm]{Remark}

\def\BE#1{\begin{equation}\label{#1}}
\def\EE{\end{equation}}
\def\eref#1{(\ref{#1})}

\def\BEnum#1{\begin{enumerate}[label=#1,leftmargin=*,topsep=-10pt,itemsep=-3pt]}
\def\EEnum{\end{enumerate}}

\def\ov#1{\overline{#1}}
\def\sf#1{\textsf{#1}}
\def\wt#1{\widetilde{#1}}
\def\tn#1{\textnormal{#1}} 
\def\lr#1{\langle{#1}\rangle}
\def\blr#1{\big\langle{#1}\big\rangle}
\def\bllrr#1{\big\langle\!\!\big\langle{#1}\big\rangle\!\!\big\rangle}

\def\unbr#1#2{\underset{#2}{\underbrace{#1}}}

\def\lra{\longrightarrow}
\def\Lra{\Longrightarrow}
\def\xlra#1{\xrightarrow{{#1}}}

\def\C{\mathbb C}
\def\cC{\mathcal C}
\def\D{\mathbb D}
\def\cD{\mathcal D}

\def\cJ{\mathcal J}

\def\M{\mathfrak M}
\def\cM{\mathcal M}
\def\cN{\mathcal N}

\def\cP{\mathcal P}
\def\R{\mathbb R}
\def\Q{\mathbb Q}
\def\bS{\mathbb S}

\def\Z{\mathbb Z}

\def\al{\alpha}
\def\be{\beta}

\def\ep{\epsilon}
\def\ga{\gamma}
\def\io{\iota}

\def\la{\lambda}
\def\si{\sigma}
\def\om{\omega}

\def\De{\Delta}
\def\Ga{\Gamma}

\def\fb{\mathfrak b}
\def\fbb{\mathfrak{bb}}
\def\fc{\mathfrak c}

\def\ff{\mathfrak f}

\def\fo{\mathfrak o}

\def\fp{\mathfrak p}
\def\fs{\mathfrak s}

\def\u{\mathbf u}

\def\Br{\tn{Br}}
\def\codim{\tn{codim}}

\def\nd{\tn{d}}
\def\dim{\tn{dim}}
\def\dom{\tn{dom}}
\def\ev{\tn{ev}}
\def\evb{\tn{evb}}
\def\evi{\tn{evi}}
\def\FPt{\tn{FPt}}
\def\FPC{\tn{FPC}}
\def\id{\tn{id}}
\def\Id{\tn{Id}}
\def\Im{\tn{Im}}
\def\lk{\tn{lk}}
\def\nod{\tn{nd}}
\def\PD{\tn{PD}}

\def\PC{\tn{PC}}
\def\pt{\tn{pt}}

\def\sgn{\tn{sgn}}

\def\fiber{\times_\tn{fb}}

\def\uo{\tn{uo}}

\def\MDC{\tn{MD}}
\def\DMDC{\tn{DMD}}
\def\SDC{\tn{SD}}
\def\ST{\tn{ST}}

\def\i{\infty}

\def\eset{\emptyset}
\def\prt{\partial}
\def\dbar{\ov\partial}
\def\st{\bigstar}

\def\os{\mathfrak{os}}

\def\bu{\bullet}
\def\v{\vee}

\begin{document}

\title{Solomon-Tukachinsky's vs.~Welschinger's\\
Open Gromov-Witten Invariants of Symplectic Sixfolds}
\author{Xujia Chen\thanks{Supported by NSF grant DMS 1901979}}
\date{\today}

\maketitle

\begin{abstract}
\noindent
Our previous paper describes a geometric translation of 
the construction of open Gromov-Witten invariants by J.~Solomon and S.~Tukachinsky
from a perspective of $A_{\i}$-algebras of differential forms.
We now use this geometric perspective to show that these invariants reduce to 
Welschinger's open Gromov-Witten invariants in dimension~6,
inline with their and G.~Tian's expectations.
As an immediate corollary, we obtain a translation of 
Solomon-Tukachinsky's open WDVV equations into relations
for Welschinger's invariants.
\end{abstract}

\tableofcontents
\setlength{\parskip}{\baselineskip}

\section{Introduction}
\label{intro_sec}

Suppose $(X,\om)$ is a compact symplectic sixfold, 
$Y\!\subset\!X$ is a compact Lagrangian submanifold,
and $\os\!\equiv\!(\fo,\fs)$ is an OSpin-structure on~$Y$,
i.e.~a pair consisting of an orientation~$\fo$ on~$Y$ and a Spin-structure~$\fs$ on
the oriented manifold~$(Y,\fo)$. 
If 
\BE{strongpos_e2}\om(\be)>0,~~\mu_Y^{\om}(\be)\ge0 \qquad\Lra\qquad
\mu_Y^{\om}(\be)>0\EE
for every $\be\!\in\!H_2(X,Y;\Z)$ representable by a map from~$(\D^2,S^1)$,
every non-constant $J$-holomorphic map from $(\D^2,S^1)$ to~$(X,Y)$ 
has a positive Maslov index for a generic $\om$-compatible almost complex structure~$J$;
the same happens along a generic path of almost complex structures.

Let $R$ be a commutative ring with unity~1.
If the homomorphism
\BE{iohomdfn_e}\io_{Y*}\!:H_1(Y;R)\lra H_1(X;R)\EE
induced by the inclusion $\io_Y\!:Y\!\lra\!X$ is injective,
the boundaries of maps from $(\D^2,S^1)$ to~$(X,Y)$
are homologically trivial in~$Y$ and thus have well-defined linking numbers with values in~$R$.
Under this injectivity assumption and the positivity condition~\eref{strongpos_e2}, 
Welschinger~\cite{Wel13} defines \sf{open Gromov-Witten invariants}
\BE{Welinvdfn_e}
\bllrr{\cdot,\ldots,\cdot}_{\be,k}^{\om,\os}\!: 
\bigoplus_{l=0}^{\i} H^{2*}(X,Y;R)^{\oplus l}\lra R, 
\quad \be\!\in\!H_2(X,Y;\Z)\!-\!\{0\},~k\!\in\!\Z^+,\EE
enumerating $J$-holomorphic multi-disks of total degree~$\be$ 
weighted by $R$-valued linking numbers of their boundaries;
see Section~4.1 in~\cite{Wel13} and~\eref{Welinvdfn_e0}.
In ``real settings" (when $Y$ is a topological component of the fixed locus of anti-symplectic
involution on~$X$), these invariants encode the real Gromov-Witten invariants 
defined in~\cite{Wel6,Wel6b,Jake}; see Remark~\ref{WelReal_rmk}.
A special case of the setting of~\cite{Wel13} is when $Y$ is an $R$-homology~$S^3$.

Based on $A_{\i}$-algebra considerations, 
K.~Fukaya~\cite{Fuk11} defines a count \hbox{$\lr{}_{\be,0}^{\om,\os}\!\in\!\Q$} 
of $J$-holomorphic degree~$\be$ disks in~$X$ with boundary in~$Y$ under the assumption
that $(X,\om)$ is a Calabi-Yau threefold and the \sf{Maslov index} 
\BE{Maslovdfn_e}\mu_{\om}\!:H_2(X,Y;\Z)\lra\Z\EE
of~$(X,Y)$ vanishes;
this count may in general depend on the $\om$-compatible almost complex structure~$J$ on~$X$.
Motivated by~\cite{Fuk11}, 
J.~Solomon and S.~Tukachinsky~\cite{JS2} use a bounding chain
of differential forms to define counts 
\BE{JSinvdfn_e}\lr{\cdot,\ldots,\cdot}_{\be,k}^{\om,\os}\!: 
\bigoplus_{l=0}^{\i}H^{2*}(X,Y;\R)^{\oplus l}\lra\R, \quad 
\be\!\in\!H_2(X,Y;\Z),~k\!\in\!\Z^{\ge0},\EE
of $J$-holomorphic disks in symplectic manifolds~$X$ of arbitrary dimension~$2n$
and show that bounding chains that are equivalent in a suitable sense define
the same counts.
They also prove that bounding chains exist and any two relevant bounding chains
are equivalent if $n$ is odd and $Y$ is an $\R$-homology sphere. 
The resulting \sf{open Gromov-Witten invariants}~\eref{JSinvdfn_e} of~$(X,Y)$
 depend on the choice of a (relative) OSpin-structure~$\os$ on~$Y$, 
but are independent of all other auxiliary choices (such as~$J$).

Well before~\cite{JS2}, G.~Tian expressed a belief that 
the construction of~\cite{Wel13} is a geometric realization 
of the algebraic considerations behind the construction of~\cite{Fuk11}.
The same sentiment is expressed in \cite[Sec~1.2.7]{JS2}.
The present paper uses the geometric interpretation of the construction of~\cite{JS2}  
described in~\cite{JakeSaraGeom}, which is applicable over any commutative ring~$R$
with unity under the positivity condition~\eref{strongpos_e2} in the case of symplectic sixfolds,  
to confirm G.~Tian's and Solomon-Tukachinsky's
expectations; see Theorem~\ref{WelOpen_thm} below.
Proposition~\ref{RDivRel_prp}, a kind of open divisor relation
which trades real codimension~1 bordered insertions at boundary marked points
for linking numbers of their boundaries with the boundaries of the disks,
provides a transition from bounding chains to linking numbers.
As an immediate consequence of Theorem~\ref{WelOpen_thm}, 
the basic structural properties and 
the WDVV-type relations for open Gromov-Witten invariants
obtained in~\cite{JS3} yield 
similar properties and relations for Welschinger's
invariants~\eref{Welinvdfn_e}; 
see Corollaries~\ref{WelOpenPrp_crl} and~\ref{WelOpenWDVV_crl}.

The positivity condition~\eref{strongpos_e2} ensures that virtual techniques are not 
necessary for the purposes of this paper.
However, the reasoning extends to general symplectic sixfolds satisfying 
the injectivity condition~\eref{strongpos_e2} via the setup of Appendix~A 
in~\cite{JakeSaraGeom} if $R\!\supset\!\Q$.
This setup is compatible with standard virtual class approaches,
such as in~\cite{LT,FO, HWZ,Pardon}.
As we only need evaluation maps from the $(J,\nu)$-spaces to be pseudocycles,
a full virtual cycle construction and gluing across all strata of 
the $(J,\nu)$-spaces are not necessary.

The author would like to thank Penka Georgieva for hosting her at 
the Institut de Math\'ematiques de Jussieu in March~2019 and
for the enlightening discussions that inspired the present paper.
She would also like to thank Aleksey Zinger for suggesting the topic and help with the exposition.

\subsection{Comparison theorem}
\label{compthm_subs}

Let $(X,\om)$ be a compact symplectic manifold of dimension~$2n$,  
$Y\!\subset\!X$ be a compact Lagrangian submanifold,
and $\os\!\equiv\!(\fo,\fs)$ be a relative OSpin-structure on~$Y$. 
Let $\al\!\equiv\!(\be,K,L)$ be a tuple consisting of $\be\!\in\!H_2(X,Y;\Z)$, 
a generic finite subset~$K$ of~$Y$, and
a generic set~$L$ of pseudocycles $\Ga_1,\ldots,\Ga_l$ to $X\!-\!Y$
representing Poincare duals of some cohomology classes $\ga_1,\ldots,\ga_l$ on~$(X,Y)$.
For a generic $\om$-compatible almost complex structure~$J$ on~$X$,
a \sf{bounding chain} $(\fb_{\al'})_{\al'\in\cC_{\om;\al}(Y)}$ on~$(\al,J)$
in the geometric perspective of~\cite{JakeSaraGeom}
is a tuple of bordered pseudocycles with certain properties specified
in the $\dim\,X\!=\!6$ case by Definition~\ref{bndch_dfn} in the present paper.
Such a tuple determines a pseudocycle~$\fbb_{\al}$
into~$Y$; see~\eref{fbbdfn_e}.
It has a well-defined degree, and we set
\BE{JSinvdfn_e2}
\blr{\ga_1,\ldots,\ga_l}_{\be,|K|+1}^{\om,\os}\equiv
\blr{L}_{\be;K}^{\om,\os}\equiv\deg\fbb_{\al}.\EE
This degree may depend on the choices of $K$, $L$, $J$, 
and $(\fb_{\al'})_{\al'\in\cC_{\om;\al}(Y)}$.
Bounding chains differing by a pseudo-isotopy of \cite[Dfn.~2.2]{JakeSaraGeom}
determine the same degrees~\eref{JSinvdfn_e2}; see \cite[Sec.~2.2]{JakeSaraGeom}.
This guarantees that the numbers in~\eref{JSinvdfn_e2} depend only on~$\om$, $\os$, $\be$, $|K|$,
and $\ga_1,\ldots,\ga_l$ if $n$ is odd and $Y$ is an $R$-homology sphere.
However, the injectivity of~\eref{iohomdfn_e} does {\it not} guarantee the existence 
of a pseudo-isotopy between a pair of bounding chains associated even
with the same $K$, $L$, and~$J$.
Nevertheless, we establish the following.

\begin{thm}\label{WelOpen_thm}
Suppose $R$ is a commutative ring with unity, $(X,\om)$ is a compact symplectic sixfold,
$Y\!\subset\!X$ is a compact Lagrangian submanifold 
so that the positivity condition~\eref{strongpos_e2} holds and 
the homomorphism~\eref{iohomdfn_e} is injective, and
$\os$ is a relative OSpin-structure on~$Y$.
Let $\be\!\in\!H_2(X,Y;\Z)$, $K$ be a generic finite subset of~$Y$,
\hbox{$L\!\equiv\!\{\Ga_1,\ldots,\Ga_l\}$} be 
a generic set of even-dimensional pseudocycles to~$X\!-\!Y$,
$\al\!\equiv\!(\be,K,L)$, and 
 $J$ be a generic $\om$-compatible almost complex structure on~$X$.
\BEnum{(W\arabic*)}

\item\label{WelBC_it} 
There exists a bounding chain $(\fb_{\al'})_{\al'\in\cC_{\om;\al}(Y)}$
on~$(\al,J)$.

\item\label{Welgen_it} If  $\ga_1,\ldots,\ga_l\!\in\!H^{2*}(X,Y;R)$
are the Poincare duals of $\Ga_1,\ldots,\Ga_l$, then
$$\bllrr{\ga_1,\ldots,\ga_l}_{\be,|K|+1}^{\om,\os}=
(-1)^{|K|+1}\blr{L}_{\be;K}^{\om,\os}$$
for any bounding chain $(\fb_{\al'})_{\al'\in\cC_{\om;\al}(Y)}$
on~$(\al,J)$.

\EEnum
\end{thm}

\vspace{.1in}

For each $\al'\!\in\!\cC_{\om;\al}(Y)$ satisfying the dimension condition 
in Definition~\ref{bndch_dfn}\ref{BCprt_it},
the right-hand side of~\eref{BCprt_e} consists of the boundaries of some maps 
from $(\D^2,S^1)$ to~$(X,Y)$; see Proposition~\ref{WelOpen_prp}.
By the injectivity of~\eref{iohomdfn_e}, we can thus choose a bordered pseudocycle~$\fb_{\al'}$ 
to~$Y$ satisfying~\eref{BCprt_e}.
This implies that a bounding chain $(\fb_{\al'})_{\al'\in\cC_{\om;\al}(Y)}$
on $(\al,J)$ can be constructed by induction 
on the partially ordered set $\cC_{\om;\al}(Y)$ and 
establishes Theorem~\ref{WelOpen_thm}\ref{WelBC_it}.

A key ingredient in the proof of Proposition~\ref{WelOpen_prp}
is the Open Divisor Relation of Proposition~\ref{RDivRel_prp}, 
which replaces real codimension~1 bordered insertions at boundary marked points
with linking numbers.
This relation is also combined with Proposition~\ref{WelOpen_prp} to obtain 
the identification of the disk counts~\eref{JSinvdfn_e2}
with Welschinger's invariants~\eref{Welinvdfn_e} stated in Theorem~\ref{WelOpen_thm}\ref{Welgen_it}.
This identification in turn implies that the numbers~\eref{JSinvdfn_e2} depend only
on~$\om$, $\os$, $\be$, $|K|$, and $\ga_1,\ldots,\ga_l$.

We denote by 
\BE{qYdfn_e}q_Y\!:H_2(X;\Z)\lra H_2(X,Y;\Z)\EE 
the natural homomorphism.
A bounding chain $(\fb_{\al})_{\al\in\cC_{\om;\al}(Y)}$ as in Definition~\ref{bndch_dfn} 
can also be used to define a count of $J$-holomorphic degree~$\be$ disks
through $|K|$ (rather than $|K|\!+\!1$) points in~$Y$ if 
\BE{nosphbubb_e} k\equiv|K|\neq0 \qquad\hbox{or}\qquad 
\be\not\in\Im\big(q_Y\!:H_2(X;\Z)\!\lra\!H_2(X,Y;\Z)\!\big);\EE
see~\eref{JSinvdfn_e2b}.
The definition of the invariants~\eref{Welinvdfn_e} in~\cite{Wel13} immediately 
extends to counts of multi-disks with $k\!=\!0$ points in~$Y$ if
$\be$ satisfies the second condition in~\eref{nosphbubb_e}.
The proof of Theorem~\ref{WelOpen_thm} can be slightly modified to cover this~case. 
It can also be readily extended to the open invariants with insertions from~$H_2(Y;R)$, 
which are defined in~\cite{Wel13}.

It is immediate from~\eref{Welinvdfn_e0} that Welschinger's open invariants
are symmetric linear functionals that satisfy an open divisor relation:
$$\bllrr{\ga,\ga_1,\ldots,\ga_l}_{\be,k}^{\om,\os}
=\lr{\ga,\be}\bllrr{\ga_1,\ldots,\ga_l}_{\be,k}^{\om,\os}
\qquad\forall\,\ga\!\in\!H^2(X,Y;R).$$
Combining Theorem~\ref{WelOpen_thm} above with the last three statements of Theorem~2.9
in~\cite{JakeSaraGeom}, 
which in turn are analogues of Proposition~2.1 in~\cite{RealWDVV3}
and Corollary~1.5 in~\cite{JS3},
we obtain below additional properties of Welschinger's open invariants.

Suppose $Y$ is connected.
The kernel of the homomorphism
$$H_2(X\!-\!Y;R)\lra H_2(X;R)$$
is then generated by the homology class~$[S(\cN_yY)]_{X-Y}$ of
a unit sphere $S(\cN_yY)$ in the fiber of~$\cN Y$ over any $y\!\in\!Y$.
We orient~$S(\cN_yY)$ as in \cite[Sec~2.5]{RealWDVV3} and
denote the image of~$[S(\cN_yY)]_{X-Y}$  under the Lefschetz Duality isomorphism
$$\PD_{X,Y}\!: H_2\big(X\!-\!Y;R\big)\stackrel{\approx}{\lra} H^4\big(X,Y;R\big)$$
by~$\eta_{X,Y}^{\circ}$.
For $B\!\in\!H_2(X;\Z)$, let
$$\lr{\cdot,\ldots,\cdot}_B^{\om}\!: \bigoplus_{l=0}^{\i}H^*(X;R)^{\oplus l}\lra R$$ 
be the standard GW-invariants of~$(X,\om)$.
We denote by $[Y]_X$ the homology class on~$X$ determined by~$Y$.

\begin{crl}\label{WelOpenPrp_crl}
Let $(X,\om,Y)$, $\os$, $\be$, $k$, and $\ga_1,\ldots,\ga_l$ be
as in Theorem~\ref{WelOpen_thm}.
If the pair $(k,\be)$ satisfies~\eref{nosphbubb_e} and $Y$ is connected,
Welschinger's open invariants~\eref{Welinvdfn_e} satisfy
the following properties.
\BEnum{(WGW\arabic*)}

\item\label{sphere_it} 
$\displaystyle\bllrr{\ga_{X,Y}^{\circ},\ga_1,\ldots,\ga_l}_{\!\be,k}^{\!\om,\os}
=-\bllrr{\ga_1,\ldots,\ga_l}_{\!\be,k+1}^{\!\om,\os}$.

\item\label{lagl_it1} 
If $k\!=\!1$ and $\ga_0\!\in\!H^3(X;R)$,
$${}\hspace{-.8in}
\lr{\ga_0,[Y]_X}\bllrr{\ga_1,\ldots,\ga_l}_{\!\be,k}^{\!\om,\os}
=-\!\!\!\sum_{B\in q_Y^{-1}(\be)}\!\!\!\!\!\!\!(-1)^{\lr{w_2(\os),B}}
\blr{\PD_X\big([Y]_X\big),\ga_0,\ga_1|_X,\ldots,\ga_l|_X}_{\!B}^{\!\om}\,.$$ 

\item\label{lagl_it2} 
If $[Y]_X\!\neq\!0$ and $k\!\ge\!2$,
then $\bllrr{\ga_1,\ldots,\ga_l}_{\!\be,k}^{\!\om,\os}\!=\!0$.

\EEnum
\end{crl}

\subsection{WDVV-type relations}
\label{WelOpenWDVV_subs}

Let $R$, $(X,\om,Y)$ and $\os$ be as in Theorem~\ref{WelOpen_thm}.
We now use this theorem to translate 
the WDVV-type relations for the open GW-invariants~\eref{JSinvdfn_e} obtained in~\cite{JS3} 
to relations for Welschinger's open invariants~\eref{Welinvdfn_e} under
the assumptions that $R$~is a field and the homomorphism
\BE{iohomdfn_e2}\io_{Y*}\!:H_2(Y;R)\lra H_2(X;R)\EE
induced by the inclusion $\io_Y\!:Y\!\lra\!X$ is trivial.
For $k\!\in\!\Z^{\ge0}$, define
$$[k]=\big\{1,\ldots,k\big\}.$$

Under the assumption~\eref{nosphbubb_e}, we extend the invariants~\eref{Welinvdfn_e}  
to the degree $\be\!=\!0$ and  inputs from $H^0(X;R)$~by
\begin{equation*}\begin{split}
\bllrr{\ga_1,\ldots,\ga_l}_{0,k}^{\om,\os}&=
\begin{cases}
-\lr{\ga_1,\pt},&\hbox{if}~(k,l)\!=\!(1,1);\\
0,&\hbox{otherwise};\end{cases}\\
\bllrr{\ga_1,\ga_2,\ldots,\ga_l,1}_{\be,k}^{\om,\os}&=
\begin{cases}
-1,&\hbox{if}~(\be,k,l)\!=\!(0,1,0);\\
0,&\hbox{otherwise}.\end{cases}
\end{split}\end{equation*}
In light of Theorem~\ref{WelOpen_thm} and the symmetry of 
the open invariants~\eref{JSinvdfn_e}, these extensions are consistent 
with (OGW2) and (OGW3) in Theorem~2.9 of~\cite{JakeSaraGeom}.
If $[Y]_X\!=\!0$ and $\ga$ is a two-dimensional pseudocycle to~$X\!-\!Y$ bounding
a pseudocycle~$\fb$ transverse to~$Y$, we define
$$\lk_{\os}(\ga)\equiv\big|\fb\!\fiber\!\io_Y\big|^{\pm}\,;$$
see Section~\ref{Fp_subs} for the sign conventions for fiber products.
This \sf{linking number} of~$\ga$ and~$Y$ with the orientation determined
by the relative OSpin-structure~$\os$ does not depend on the choice of~$\fb$.
We set $\lk_{\os}(\ga)\!=\!0$ if $\ga$ is not a two-dimensional pseudocycle.

For $l\!\in\!\Z^{\ge0}$, $B\!\in\!H_2(X;\Z)$, and an $\om$-tame almost complex 
structure~$J$, we denote~by $\M^\C_{\{0\}\sqcup[l]}(B;J)$
the moduli space of stable $J$-holomorphic degree~$B$ maps with marked points
indexed by the set $\{0\}\!\sqcup\![l]$.
It carries a canonical orientation.
For each $i\!\in\!\{0\}\!\sqcup\![l]$, let 
$$\ev_i\!:\M^\C_{\{0\}\sqcup[l]}(B;J)\lra X$$
be the evaluation morphism at the $i$-th marked point.
If in addition $\Ga_1,\ldots,\Ga_l$ are maps to~$X$, let
$$\M^\C_{0\sqcup[l]}(B;J)\!\fiber\!\big(\!(i,\Ga_i)_{i\in[l]}\big)
\equiv \M^\C_{0\sqcup[l]}(B;J) _{(\ev_1,\ldots,\ev_l)}\!\!\times\!\!
_{\Ga_1\times\ldots\times\Ga_l}\!\big(\!(\dom\,\Ga_1)\!\times\!\ldots\!\times\!(\dom\,\Ga_l)\!\big).$$
If $J$ is generic and $\Ga_1,\ldots,\Ga_l$ are pseudocycles in general position, then
$$f^\C_{B,(\Ga_i)_{i\in[l]}}\equiv
\Big(\ev_0\!:\M^\C_{0\sqcup[l]}(B;J)\!\!\fiber\!\!\big(\!(i,\Ga_i)_{i\in[l]}\big)\lra X\Big)$$
is a pseudocycle of dimension 
$$\dim\,f^\C_{B,(\Ga_i)_{i\in[l]}}
=\mu_{\om}\big(q_Y(B)\!\big)\!-\!\sum_{i=1}^l\!\big(\codim\,\Ga_i\!-\!2\big)+2$$
transverse to~$Y$.
 
Since $\dim\,Y\!=\!3$, the cohomology long exact sequence for the pair~$(X,Y)$
implies that the restriction homomorphism
\BE{HpXYrel_e}H^p(X,Y;R)\lra H^p(X;R)\EE
is surjective for $p\!=\!4,6$.
Since~$R$ is a field and the homomorphism~\eref{iohomdfn_e2} is trivial,
\eref{HpXYrel_e} is  also surjective for $p\!=\!2$.
Let 
$$\ga_1^{\st}\!\equiv\!1\in H^0(X;R) \qquad\hbox{and}\qquad
\ga_2^{\st},\ldots,\ga_N^{\st}\in H^{2*}(X,Y;R)$$
be such that $\ga_1^{\st},\ga_2^{\st}|_X,\ldots,\ga_N^{\st}|_X$ is a basis for $H^{2*}(X;R)$,
$(g_{ij})_{i,j}$ be the $N\!\times\!N$-matrix given~by
$$g_{ij}=\blr{\ga_i^{\st}\ga_j^{\st},[X]}$$
and $(g^{ij})_{i,j}$ be its inverse.
Let $\Ga_1^{\st}\!=\!\id_X$ and $\Ga_2^{\st},\ldots,\Ga_N^{\st}$ be pseudocycles to~$X\!-\!Y$
representing the Poincare duals of $\ga_2^{\st},\ldots,\ga_N^{\st}$.

For the purpose of WDVV-type equations for the invariants~\eref{Welinvdfn_e},
we extend the signed counts~\eref{Welinvdfn_e0} to the pairs $(k,\be)$ not 
satisfying~\eref{nosphbubb_e},
i.e.~$k\!=\!0$ and $\be\!\in\!H_2(X,Y;\Z)$ is in the image of 
the homomorphism~$q_Y$ in~\eref{qYdfn_e},  as follows.
Let $\ga_1,\ldots,\ga_l$ be elements of $\{1\}\!\sqcup\!H^{2*}(X,Y;R)$.
If $[Y]_X\!\neq\!0$, we~define
$$\bllrr{\ga_1,\ldots,\ga_l}_{\be,0}^{\om,\os}=0.$$
Suppose next that $[Y]_X\!=\!0$.
Let $\Ga_1,\ldots,\Ga_l$ be generic pseudocycles to~$X$ so that $\Ga_i\!=\!\id_X$ 
if $\ga_i\!=\!1$ and $\Ga_i$ is a pseudocycle into~$X\!-\!Y$ 
representing the Poincare dual of~$\ga_i$ otherwise.
For $B\!\in\!H_2(X;\Z)$,
let $(\la_{B,(\ga_i)_{i\in[l]}}^j)_{j\in[N]}\!\in\!R^N$ be such that
$$\big[f^\C_{B,(\Ga_i)_{i\in[l]}}\big]
=\sum_{j=1}^N\la_{B,(\ga_i)_{i\in[l]}}^j
\PD_X\big(\ga_j^{\st}|_X\big)\in H_*(X;R);$$
the tuple $(\la_{B,(\ga_i)_{i\in[l]}}^j)_{j\in[N]}$ depends
only on $B$, $\ga_1,\ldots,\ga_l$, and $\ga_2^{\st},\ldots,\ga_N^{\st}$.
Define
$$\bllrr{\ga_1,\ldots,\ga_l}_{\be,0}^{\om,\os}=
\hbox{RHS of~\eref{Welinvdfn_e0}}~+
\sum_{B\in q_Y^{-1}(\be)}\!\!\!\!\!\!(-1)^{\blr{w_2(\os),B}}
\lk_{\os}\!\Big(f^\C_{B,(\Ga_i)_{i\in[l]}}-
\sum_{j=1}^N\!\la_{B,(\ga_i)_{i\in[l]}}^j\Ga_j^{\st}\Big)$$
in this case.
This number depends on the span of
the elements $\ga_2^{\st},\ldots,\ga_N^{\st}$ in $H^{2*}(X,Y;R)$,
but not on the choice of pseudocycles $\Ga_1,\ldots,\Ga_l$ and
$\Ga_2^{\st},\ldots,\Ga_N^{\st}$ representing 
the Poincare duals of $\ga_1,\ldots,\ga_l$ and $\ga_2^{\st},\ldots,\ga_N^{\st}$,
respectively.
For example,
$$\bllrr{\ga_1,\ga_2}_{0,0}^{\om,\os}=
\lk_{\os}\!\Big(\Ga_1\!\cap\!\Ga_2\!-\!
\sum_{j=1}^N\!\la_{\ga_1\ga_2}^j\Ga_j^{\st}\Big), 
\quad\hbox{where}~~
\ga_1\ga_2\equiv\sum_{j=1}^N\la_{\ga_1\ga_2}^j\ga_j^{\st}|_X\in H^*(X;\Q).$$

Let $\ga\!\equiv\!(\ga_1,\ldots,\ga_l)$ be a tuple of elements 
of $\{1\}\!\sqcup\!H^{2*}(X,Y;R)$.
For $I\!\subset\!\{1,2,\ldots,l\}$, we denote by~$\ga_I$ 
the $|I|$-tuple consisting of the entries of~$\ga$
indexed by~$I$.
If in addition $\be\!\in\!H_2(X,Y;\Z)$, define
$$k_{\be}(\ga_I)
\equiv\frac12\Big(\mu_{\om}(\be)-\sum_{i\in I}\!\big(\!\deg\ga_i\!-\!2\big)\!\!\Big),
\quad
\bllrr{\ga_I}_{\!\be}^{\!\om,\os}=
\begin{cases}\bllrr{\ga_I}_{\be,k_{\be}(\ga_I)}^{\om,\os},&
\hbox{if}~k_{\be}(\ga_I)\!\ge\!0;\\
0,&\hbox{otherwise}.\end{cases}$$

For $i,j\!=\!1,2,\ldots,l$, we define
\begin{alignat*}{2}
\cP(l)&=\big\{(I,J)\!:\{1,2,\ldots,l\}\!=\!I\!\sqcup\!J,~1\!\in\!I\big\},&\quad
\cP_{i;}(l)&=\big\{(I,J)\!\in\!\cP(l)\!:i\!\in\!I\big\},\\
\cP_{;j}(l)&=\big\{(I,J)\!\in\!\cP(l)\!:j\!\in\!J\big\},&\quad
\cP_{i;j}(l)&=\cP_{i;}(l)\!\cap\!\cP_{;j}(l).
\end{alignat*}
For $\be\!\in\!H_2(X,Y;\Z)$, let
\begin{equation*}\begin{split}
\cP_{\C}(\be)&=\big\{(\be',B)\!\in\!H_2(X,Y;\Z)\!\oplus\!H_2(X;\Z)\!:
\,\be'\!+\!q_Y(B)\!=\!\be\big\},\\
\cP_{\R}(\be)&=\big\{(\be_1,\be_2)\!\in\!H_2(X,Y;\Z)\!\oplus\!H_2(X,Y;\Z)\!:
\,\be_1\!+\!\be_2\!=\!\be\big\}.
\end{split}\end{equation*}
Combining Theorem~\ref{WelOpen_thm} above with Theorem~2.10
in~\cite{JakeSaraGeom}, we obtain relations between 
Welschinger's open invariants~\eref{Welinvdfn_e} as stated below.

\begin{crl}\label{WelOpenWDVV_crl}
Let $R$ be a field and $(X,\om,Y)$, $\os$, $\be$, and 
$\ga\!\equiv\!(\ga_1,\ldots,\ga_l)$ be as in Theorem~\ref{WelOpen_thm} with
$$k\equiv\frac12\Big(\mu_{\om}(\be)-\sum_{i=1}^l\!\!\big(\!\deg\mu_i\!-\!2\big)\!\!\Big)
\!-\!1\ge0.$$
Suppose in addition that the homomorphism~\eref{iohomdfn_e2} is trivial.
\BEnum{($\R$WDVV\arabic*)}

\item\label{Wel12rec3_it} If $l\!\ge\!2$ and $k\!\ge\!1$, then
\begin{equation*}\begin{split}
{}\hspace{-1in}
\sum_{\begin{subarray}{c}(\be',B)\in\cP_{\C}(\be)\\ 
(I,J)\in\cP_{2;}(l)\end{subarray}}
\sum_{i,j\in[N]}\!\!\!\!
\blr{\ga_I|_X,\ga^\st_i|_X}^{\!\om}_{\!B}g^{ij}\bllrr{\ga^\st_j,\ga_J}^{\!\om,\os}_{\!\be'}
&-\sum_{\begin{subarray}{c}(\be_1,\be_2)\in\cP_{\R}(\be)\\
(I,J)\in\cP_{2;}(l)\end{subarray}}
\hspace{-.08in}\binom{k\!-\!1}{k_{\be_1}(\ga_I)}
\bllrr{\ga_I}^{\!\om,\os}_{\!\be_1}\bllrr{\ga_J}^{\!\om,\os}_{\!\be_2}\\
&=-\!\!\sum_{\begin{subarray}{c}(\be_1,\be_2)\in\cP_{\R}(\be)\\
(I,J)\in\cP_{;2}(l)\end{subarray}}
\hspace{-.08in}
\binom{k\!-\!1}{k_{\be_1}(\ga_I)\!-\!1}
\bllrr{\ga_I}^{\!\om,\os}_{\!\be_1}
\bllrr{\ga_J}^{\!\om,\os}_{\!\be_2}\,.
\end{split}\end{equation*}

\item\label{Wel03rec3_it} If $l\!\ge\!3$, then
\begin{equation*}\begin{split}
&\hspace{-1in}
\sum_{\begin{subarray}{c}(\be',B)\in\cP_{\C}(\be)\\ 
(I,J)\in\cP_{2;3}(l)\end{subarray}}
\sum_{i,j\in[N]}\!\!\!\!
\blr{\ga_I|_X,\ga^\st_i|_X}^{\om}_{\!B}g^{ij}
\bllrr{\ga^{\st}_j,\ga_J}^{\!\om,\os}_{\!\be'}
-\sum_{\begin{subarray}{c}(\be_1,\be_2)\in\cP_{\R}(\be)\\ 
(I,J)\in\cP_{2;3}(l)\end{subarray}}
\hspace{-.1in}
\binom{k}{k_{\be_1}(\ga_I)}
\bllrr{\ga_I}^{\!\om,\os}_{\!\be_1}\bllrr{\ga_J}^{\!\om,\os}_{\!\be_2}\\
&\hspace{-.8in}
=\sum_{\begin{subarray}{c}(\be',B)\in\cP_{\C}(\be)\\ 
(I,J)\in\cP_{3;2}(l)\end{subarray}}
\sum_{i,j\in[N]}\!\!\!\!
\blr{\ga_I|_X,\ga^\st_i|_X}^{\!\om}_{\!B}g^{ij}
\bllrr{\ga^{\st}_j,\ga_J}^{\!\om,\os}_{\!\be'}
-\!\!\sum_{\begin{subarray}{c}(\be_1,\be_2)\in\cP_{\R}(\be)\\ 
(I,J)\in\cP_{3;2}(l)\end{subarray}}
\hspace{-.08in}
\binom{k}{k_{\be_1}(\ga_I)}
\bllrr{\ga_I}^{\!\om,\os}_{\!\be_1}\bllrr{\ga_J}^{\!\om,\os}_{\!\be_2}\,.
\end{split}\end{equation*}

\EEnum
\end{crl}

\section{Preliminaries}
\label{prelim_sec}

Section~\ref{Fp_subs} recalls the orientation conventions for fiber products and 
some of their properties from \cite[Sec.~5.1]{JakeSaraGeom}.
The combinatorial objects needed for the geometric presentation of 
the open invariants of~\cite{JS2} in symplectic sixfolds is gathered in 
Section~\ref{notation_subs}.
We describe the relevant moduli spaces of stable disk maps
and specify their orientations in Section~\ref{Ms_subs}.
Section~\ref{OpenGWs_subs} specializes the geometric definition of bounding chain 
from~\cite{JakeSaraGeom} to symplectic sixfolds and uses it to define
counts $J$-holomorphic disks.

\subsection{Fiber products}
\label{Fp_subs}

We say a short exact sequence of oriented vector spaces
$$0\lra V'\lra V\lra V''\lra0$$
is \sf{orientation-compatible} if for an oriented basis $(v'_1,\ldots,v'_m)$ of $V'$, 
an oriented basis $(v''_1,\ldots,v''_n)$ of $V''$, 
and a splitting $j\!:V''\!\lra\!V$, 
$(v'_1,\ldots,v'_m,j(v''_1),\ldots,j(v''_n)\!)$ is an oriented basis of~$V$. 
We say it has sign $(-1)^\ep$ if it becomes orientation-compatible 
after twisting the orientation of $V$ by $(-1)^\ep$. 
We use the analogous terminology for
short exact sequences of Fredholm operators 
with respect to orientations of their determinants;
see \cite[Section~2]{detLB}.

Let $M$ be an oriented manifold with boundary $\prt M$. We orient the normal bundle 
$\cN$ to $\prt M$ by the outer normal direction and orient $\prt M$ so that the short exact sequence  
$$0\lra T_p\prt M\lra T_pM\lra \cN\lra0$$
is orientation-compatible at each point $p\in\prt M$. 
We refer to this orientation of $\prt M$ as the \sf{boundary orientation}. 

We orient $M\!\times\!M$ by the usual product orientation and 
the diagonal $\De_M\subset M\!\times\!M$ by the diffeomorphism 
$$M\lra\De_M, \qquad p\lra(p,p).$$
We orient the normal bundle $\cN\De_M$ of $\De_M$ so that the short exact sequence 
$$0\lra T_{(p,p)}\De_M\lra T_{(p,p)}(M\!\times\!M)\lra \cN\De_M|_{(p,p)}\lra0$$
is orientation-compatible for each point $p\!\in\!M$. 
Thus, the isomorphism
$$\cN\De_M|_{(p,p)} \lra T_pM, \qquad [v,w]\lra w\!-\!v,$$
respects the orientations.

For maps $f\!:M\!\lra\!X$ and $g\!:\Ga\!\lra\!X$, we denote by 
$$f\!\!\fiber\!g\equiv
M_f\!\times_g\!\Ga\equiv\{(p,q)\!\in\!M\!\times\!\Ga\!:\,f(p)=g(q)\}$$
their fiber product. 
If $M,\Ga$, and $X$ are oriented manifolds ($M,\Ga$ possibly with boundary) and 
$f,f|_{\prt M}$ are transverse to $g,g|_{\prt\Ga}$, 
we orient $M_f\!\times_g\!\Ga$ so that the short exact sequence
$$0\lra T_{(p,q)}(M_f\!\times_g\!\Ga)\lra T_{(p,q)}(M\!\times\!\Ga)
\xrightarrow{[d_p\!f,d_qg]} \cN\De_X|_{(f(p),g(q))}\lra0$$
is orientation-compatible for every $(p,q)\!\in\!M_f\!\times_g\!\Ga$. 
The exact sequence 
$$0\lra T_{(p,q)}(M_f\!\times_g\!\Ga)\lra 
T_{(p,q)}(M\!\times\!\Ga)\xrightarrow{d_qg-d_pf}T_{f(p)}X\lra0$$
is then orientation-compatible as well. 
We refer to this orientation of $M_f\!\times_g\!\Ga$ as 
the \sf{fiber product orientation}. 

\begin{lmm}\label{fibersign_lmm}
With the assumptions as above,
$$\prt(M_f\!\times_g\!\Ga)=(-1)^{\dim\,X}\bigg(\!
(-1)^{\dim\,\Ga}(\prt M)_f\!\times_g\!\Ga
\sqcup M_f\!\times_g\!\Ga\bigg).$$
\end{lmm}

For a diffeomorphism $\si\!:M\!\lra\!M$ between oriented manifolds,
we define $\sgn\,\si\!=\!1$ if $\si$ is everywhere orientation-preserving 
and $\sgn\,\si\!=\!-1$ if $\si$ is everywhere orientation-reversing;
this notion is also well-defined if $M$ is orientable and $\si$ preserves each connected
component of~$M$.

\begin{lmm}\label{fibprodflip_lmm}
Suppose $\si_M,\si_{\Ga},\si_X$ are diffeomorphisms of $M,\Ga,X$, respectively,
with well-defined signs.
If the diagram 
$$\xymatrix{M\ar[rr]^f \ar[d]_{\si_M}&& X \ar[d]_{\si_X} && \Ga\ar[ll]_g\ar[d]^{\si_{\Ga}}\\
M\ar[rr]^f&& X  && \Ga\ar[ll]_g}$$
commutes, then the sign of the diffeomorphism
$$M_f\!\!\times_g\!\Ga\lra M_f\!\!\times_g\!\Ga, \qquad
(p,q)\lra \big(\si_M(p),\si_{\Ga}(q)\!\big),$$
is $(\sgn\,\si_M)(\sgn\,\si_{\Ga})(\sgn\,\si_X)$.
\end{lmm}

Let $M,\Ga,X$ and $f,g$ be as above Lemma~\ref{fibersign_lmm}.
Suppose in addition that $e\!:M\!\lra\!Y$ and $h\!:C\!\lra\!Y$.
Let \hbox{$e'\!:M_f\!\times_g\!\Ga\!\lra\!Y$} be the map induced by~$e$;
see Figure~\ref{fibprodisom_fig1}.
There is then a natural bijection
\BE{fibprodisom_e1} \big(M_f\!\times_g\!\Ga\big)\!\,_{e'}\!\!\times_h\!C \approx  
M_{(f,e)}\!\!\times_{g\times h}\!(\Ga\!\times\!C)\,.\EE
If $C,Y$ are oriented manifolds and all relevant maps are transverse,
then both sides of this bijection inherit fiber product orientations.
For any map $h\!:M\!\lra\!Z$ between manifolds, let
$$\codim\,h=\dim\,Z-\dim\,M\,.$$

\vspace{-.15in}

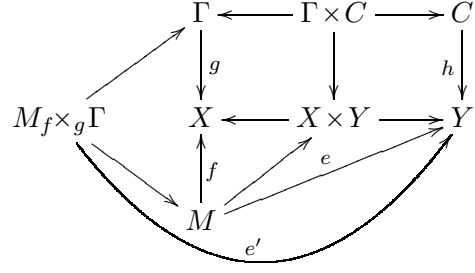
\begin{figure}
$$\xymatrix{& \Ga\ar[d]^g& \Ga\!\times\!C \ar[d]\ar[l]\ar[r]& C\ar[d]_h& \\
M_f\!\!\times_g\!\Ga \ar[ur]\ar[dr] \ar@/_4.5pc/[rrr]^{\!\!\!\!e'}& 
X&  X\!\times\!Y \ar[l]\ar[r] & Y\\
& M\ar[u]_{\!f}\ar[ur]\ar[urr]^{\!e}}$$
\caption{The maps of Lemma~\ref{fibprodisom_lmm1}.}
\label{fibprodisom_fig1}
\end{figure}

\begin{lmm}\label{fibprodisom_lmm1}
The diffeomorphism~\eref{fibprodisom_e1} has sign $(-1)^{(\dim\,X)(\codim\,h)}$ with respect 
to the fiber product orientations on the two sides.
\end{lmm}


\subsection{Combinatorial notation}
\label{notation_subs}

Let $(X,\om)$ be a compact symplectic sixfold,
$Y\!\subset\!X$ be a compact Lagrangian submanifold, 
$$H_2^{\om}(X,Y)=\big\{\be\!\in\!H_2(X,Y;\Z)\!:\om(\be)\!>\!0~\hbox{or}~\be\!=\!0\big\},$$
and $\cJ_{\om}$ be the space of $\om$-compatible almost complex structures on~$X$.
We denote by $\PC(X\!-\!Y)$ the collection of pseudocycles to~$X\!-\!Y$ with coefficients in~$R$
as defined in \cite[Sec~1]{pseudo},
by $\FPC(X\!-\!Y)$ the collection of finite subsets of~$\PC(X\!-\!Y)$, 
and by $\FPt(Y)$ the collection of finite subsets of~$Y$.
Let 
$$\cC_{\om}(Y)=\big\{(\be,K,L)\!:\,\be\!\in\!H_2^{\om}(X,Y),\,
K\!\in\!\FPt(Y),\,L\!\in\!\FPC(X\!-\!Y),\,
(\be,K,L)\!\neq\!(0,\eset,\eset)\big\}.$$
This collection has a natural partial order:
$$(\be',K',L')\preceq(\be,K,L) \qquad\hbox{if}\qquad
\be\!-\!\be'\in H_2^{\om}(X,Y), \quad
K'\!\subset\!K, \quad\hbox{and}\quad L'\!\subset\!L.$$
The elements $(0,K,L)$ of $\cC_{\om}(Y)$ with $|K|\!+\!|L|\!=\!1$ are minimal 
with respect to this partial order.
For each element $\al\!\equiv\!(\be,K,L)$ of~$\cC_\om(Y)$, we define 
\begin{gather*}
\be(\al)\equiv\be,\qquad K(\al)\equiv K,\qquad L(\al)\equiv L, \\
\dim(\al)
=\mu_Y^{\om}(\be)\!-\!2|K|\!-\!\sum_{\Ga\in L}\!\!\big(\codim\,\Ga\!-\!2\big),\quad
\cC_{\om;\al}(Y)=\big\{\al'\!\in\!\cC_{\om}(Y)\!:\,\al'\!\prec\!\al\big\}.
\end{gather*}

\vspace{-.15in}

For $\al\!\in\!\cC_{\om}(Y)$, let
\begin{equation*}\begin{split}
\cD_{\om}(\al)=\bigg\{\!\Big(\be_{\bu},k_{\bu},L_{\bu},(\al_i)_{i\in[k_{\bu}]}\Big)
\!\!:\be_{\bu}\!\in\!H_2^{\om}(X,Y),\,k_{\bu}\!\in\!\Z^{\ge0},\,L_{\bu}\!\subset\!L(\al),\,
(\be_{\bu},k_{\bu},L_{\bu})\!\neq\!(0,1,\eset),&\\
\al_i\!\in\!\cC_{\om}(Y)\,\forall\,i\!\in\![k_{\bu}],\,
\be_{\bu}\!+\!\sum_{i=1}^{k_{\bu}}\be(\al_i)\!=\!\be(\al),
~\bigsqcup_{i=1}^{k_{\bu}}\!K(\al_i)\!=\!K(\al),~
L_{\bu}\!\sqcup\!\bigsqcup_{i=1}^{k_{\bu}}\!L_i(\al)\!=\!L(\al)&\bigg\}\,.
\end{split}\end{equation*}
Since $\al_i\!\prec\!\al$ for every  
\BE{degenelemdfn_e} \eta\equiv\big(\be_{\bu},k_{\bu},L_{\bu},(\al_i)_{i\in[k_{\bu}]}\big)
\equiv\big(\be_{\bu},k_{\bu},L_{\bu},(\be_i,K_i,L_i)_{i\in[k_{\bu}]}\big)
\in \cD_{\om}(\al)\EE
and every $i\!\in\![k_{\bu}]$, $k_{\bu}\!=\!0$ if $\al$ is a minimal element 
of~$\cC_{\om}(Y)$.
Thus,
$$\cD_{\om}\big(0,\{\pt\},\eset\big)=\eset~~\forall\,\pt\!\in\!Y
\quad\hbox{and}\quad
\cD_{\om}\big(0,\eset,\{\Ga\}\big)=\big\{\big(0,0,\{\Ga\},()\!\big)\big\}
~~\forall\,\Ga\!\in\!\PC(X\!-\!Y)\,.$$
For $\eta\!\in\!\cD_{\om}(\al)$ as in~\eref{degenelemdfn_e} and $i\!\in\![k_{\bu}]$, 
we define 
\begin{gather*}
\be_{\bu}(\eta)=\be_{\bu},\quad k_{\bu}(\eta)=k_{\bu}, \quad L_{\bu}(\eta)=L_{\bu}, \\
\be_i(\eta)=\be_i,\quad K_i(\eta)=K_i, \quad L_i(\eta)=L_i, \quad
\al_i(\eta)=\al_i=(\be_i,K_i,L_i).
\end{gather*}

\subsection{Moduli spaces}
\label{Ms_subs}

We denote by $\D^2\!\subset\!\C$ the unit disk with the induced complex structure,
by  $\D^2\!\v\!\D^2$ the union of two disks joined at a pair of boundary points,
and by $S^1\!\subset\!\D^2$ and $S^1\!\v\!S^1\!\subset\!\D^2\!\v\!\D^2$ 
the respective boundaries.
We orient the boundaries counterclockwise; 
thus, starting from a smooth point~$x_0$ of $S^1\!\v\!S^1$, 
we proceed counterclockwise to the node~$\nod$,
then circle the second copy of~$S^1$ counterclockwise back to~$\nod$,
and return to~$x_0$ counterclockwise from~$\nod$.
We call smooth points $x_0,x_1,\ldots,x_k$ on $S^1$ or $S^1\!\v\!S^1$ \sf{ordered by position}
if they are traversed counterclockwise.

Let $k,l\!\in\!\Z^{\ge0}$ with $k\!+\!2l\!\ge\!3$.
We denote by $\cM_{k,l}^{\uo}$ be the moduli space of $k$ distinct boundary marked points 
$x_1,\ldots,x_k$ and $l$ distinct interior marked points $z_1,\ldots,z_l$ on the unit disk~$\D$
(the superscript~{\it uo} refers to the boundary marked points being {\it unordered} 
with respect to their position on $S^1\!\subset\!\D^2$ relative to the order on~$[k]$).
We orient $\cM_{1,1}^{\uo}$ as a plus point.
The space~$\cM_{3,0}^{\uo}$ consists of two points, 
$\cC_{3,0}^+$ with the three boundary points ordered by position and
$\cC_{3,0}^-$ with the three boundary points not ordered by position.
We orient $\cC_{3,0}^+$ as a plus point and $\cC_{3,0}^-$ as a minus point.
We identify $\cM_{0,2}^{\uo}$ with the interval $(0,1)$ by taking 
$z_1\!=\!0$ and $z_2\!\in\!(0,1)$ and orient it by the negative orientation of~$(0,1)$. 

We orient other $\cM_{k,l}^{\uo}$ inductively. 
If $k\!\ge\!1$, we orient $\cM_{k,l}^{\uo}$ so that the short exact sequence
\BE{cml_e}0\lra T_{x_k}S^1\lra T\cM_{k,l}^{\uo}
\xlra{\nd\ff^\R_k} T\cM_{k-1,l}^{\uo}\lra0\EE
induced by the forgetful morphism $\ff^\R_k$ dropping $x_k$ has sign $(-1)^k$ 
with respect to the counter-clockwise orientation of~$S^1$. 
Thus, 
$$T\cM_{k,l}^{\uo}\approx T\cM_{k-1,l}^{\uo}\oplus T_{x_k}S^1.$$
If $l\!\ge\!1$, we orient $\cM_{k,l}^{\uo}$ so that the short exact sequence
\BE{cmk_e}0\lra T_{z_l}\D\lra T\cM_{k,l}^{\uo}
\xlra{\nd\ff^\C_l} T\cM_{k,l-1}^{\uo}\lra0\EE
induced by the forgetful morphism $\ff^\C_l$ dropping $z_l$ 
is orientation-compatible with respect to the complex orientation of~$\D$. 
By a direct check, the orientations of $\cM_{1,2}^{\uo}$ induced from 
$\cM_{0,2}^{\uo}$ via~\eref{cml_e} and from $\cM_{1,1}^{\uo}$ via~\eref{cmk_e} are the same, 
and the orientations of $\cM_{3,1}^{\uo}$ induced from~$\cM_{1,1}^{\uo}$
via~\eref{cml_e} and from~$\cM_{3,0}^{\uo}$  via~\eref{cmk_e} are also the same. 
Since the fibers of $\ff^{\C}_l$ are even-dimensional, 
it follows that the orientation on~$\cM_{k,l}^{\uo}$ above is well-defined. 

Let $(X,\om)$ be a symplectic manifold, $Y\!\subset\!X$ be a Lagrangian submanifold,
$\be\!\in\!H^\om_2(X,Y)$, and $J\!\in\!\cJ_\om$.
For a finite ordered set~$K$ and a finite set~$L$, we denote by 
$\M_{K,L}^{\uo,\st}(\be;J)$
the moduli space of stable $J$-holomorphic degree~$\be$ maps from
$(\D^2,S^1)$ and $(\D^2\!\v\!\D^2,S^1\!\v\!S^1)$ to~$(X,Y)$ with 
boundary and interior marked points indexed by~$K$ and~$L$, respectively.
Let 
$$\M_{K,L}^{\uo}(\be;J)\subset \M_{K,L}^{\uo,\st}\!(\be;J)$$
be the subspace of maps from~$(\D^2,S^1)$.
If $K\!=\![k]$ for $k\!\in\!\Z^{\ge0}$ (resp.~$L\!=\![l]$ for $l\!\in\!\Z^{\ge0}$),
we write~$k$ for~$K$ (resp.~$l$ for~$L$) in the subscripts of these moduli spaces. 
For 
\BE{udfn_e} [\u]\equiv\big[u\!:(\D,S^1)\!\lra\!(X,Y),(x_i)_{i\in[k]},(z_i)_{i\in[l]}\big]
\in \M_{k,l}^{\uo}(\be;J)\,,\EE
let
$$D_{J;\u}\!:\Ga\big(u^*TX,u|_{S^1}^*TY\big)\lra
\Ga\big(T^*\D^{0,1}\!\otimes_{\C}\!u^*(TX,J)\big)$$
be the linearization of the $\{\dbar_J\}$-operator on the space of maps from~$(\D,S^1)$ to $(X,Y)$. 
By Proposition~8.1.1 in~\cite{FOOO}, 
a relative OSpin-structure~$\os$ on~$Y$ determines an orientation on $\det(D_{J;\u})$. 

We orient $\M^{\uo}_{k,l}(\be;J)$ by requiring the short exact sequence 
$$0\lra \ker D_{J;\u}\lra T_{\u}\M_{k,l}^{\uo}(\be;J)
\xlra{\nd\ff} T_{\ff(\u)}\cM_{k,l}^{\uo}\lra0$$
to be orientation-compatible, where $\ff$ is the forgetful morphism dropping the map part of~$\u$.
This orientation extends over $\M^{\uo,\st}_{k,l}\!(\be;J)$.
If $K$ is a finite ordered set and $L$ is a finite set,
we orient $\M_{K,L}^{\uo,\st}\!(\be;J)$ from $\M_{|K|,|L|}^{\uo,\st}(\be;J)$
by identifying~$K$ with~$[|K|]$ as ordered sets and~$L$ with~$[|L|]$ as~sets.

\vspace{.1in}

\begin{rmk}\label{Morient_rmk}
The above paragraph endows $\M^{\uo,\st}_{K,L}\!(\be;J)$ with an orientation under
the assumption that $|K|\!+\!2|L|\!\ge\!3$.
If $|K|\!+\!2|L|\!<\!3$, one first stabilizes the domain of~$\u$ by adding one or 
two interior marked points, then orients 
the tangent space of the resulting map as above, and 
finally drops the added marked points using the canonical complex orientation of~$\D$;
see the proof of Corollary~1.8 in~\cite{Penka1}.
\end{rmk}

\vspace{-.1in}

For $i\!\in\!K$ and $i\!\in\!L$, let
$$\evb_i\!:\M_{K,L}^{\uo,\st}(\be;J)\!\lra Y \qquad\hbox{and}\qquad
\evi_i\!:\M_{K,L}^{\uo,\st}\!(\be;J)\lra X$$
be the evaluation morphisms at the $i$-th boundary marked point and 
the $i$-th interior marked point, respectively. 
If $M\!\subset\!\M_{K,L}^{\uo,\st}(\be;J)$, we denote the restrictions of 
$\evb_i$ and $\evi_i$ to $M$ also by~$\evb_i$ and~$\evi_i$. 
If in addition $m,m'\!\in\!\Z^{\ge0}$, 
$$\big(\fb_s\!:Z_{\fb_s}\!\lra\!Y\big)_{\!s\in[m]}
\qquad\hbox{and}\qquad
\big(\Ga_s\!:Z_{\Ga_s}\!\lra\!X\big)_{\!s\in [m']}$$
are tuples of maps and $i_1,\ldots,i_m\!\in\![k]$ and $j_1,\ldots,j_{m'}\!\in\!L$ are
distinct elements, 
let
\begin{equation*}\begin{split}
&M\!\!\fiber\!\big(\!(i_s,\fb_s)_{s\in[m]};(j_s,\Ga_s)_{s\in[m']}\big)\\
&~\equiv 
M_{(\evb_{i_1},\ldots,\evb_{i_m},\evi_{j_1},\ldots,\evi_{j_{m'}})}\!\!
\times_{\fb_1\times\ldots\times\fb_m\times\Ga_1\times\ldots\times\Ga_{m'}}\!\!
\big(Z_{\fb_1}\!\times\!\ldots\!\times\!Z_{\fb_m}\!\times\!
Z_{\Ga_1}\!\times\!\ldots\!\times\!Z_{\Ga_{m'}}\!\big)
\end{split}\end{equation*} 
be their fiber product with~$M$. 
If $M$ is an oriented manifold and~$\fb_s$ and~$\Ga_s$ are smooth maps 
from oriented manifolds satisfying the appropriate transversality conditions, 
then we orient this space as in Section~\ref{Fp_subs}. 
For $i\!\in\![k]$ with $i\!\neq\!i_s$ for any $s\!\in\![m]$ 
(resp. $i\!\in\!L$ with $i\!\neq\!j_s$ for any $s\!\in\![m']$), 
we define
$$\evb_i \ (\textnormal{resp. }\evi_i)\!:
M\!\fiber\!\big(\!(i_s,\fb_s)_{s\in[m]};(j_s,\Ga_s)_{s\in[m']}\big)
\lra Y\ (\tn{resp. }X)$$
to be the composition of the evaluation map~$\evb_i$ (resp. $\evi_i$) defined above with 
the projection to the first component.

\subsection{Open Gromov-Witten invariants}
\label{OpenGWs_subs}

In the remainder of this paper, we use the terms (\sf{bordered}) $\Z_2$-\sf{pseudocycle}
and \sf{pseudocycle} to mean the respective pseudocycles in the usual sense
taken with~$R$-coefficients;
see the last part of Section~3 in~\cite{RealWDVV3} for precise definitions.
We recall that every $R$-homology class in a manifold can be represented
by a pseudocycle in this sense, which is unique up to equivalence;
see Theorem~1.1 in~\cite{pseudo}.

Let $(X,\om)$ be a symplectic sixfold and $Y\!\subset\!X$ be a Lagrangian submanifold.
For a point $\pt\!\in\!Y$, we denote its inclusion into~$Y$ also by~$\pt$.
For $\be\!\in\!H^\om_2(X,Y)$, $k\!\in\!\Z^{\ge0}$, a finite set~$L$,
and $J\!\in\!\cJ_\om$, let
$$\M_{k,L}^{\st}(\be;J)\subset\M_{k,L}^{\uo,\st}(\be;J)$$ 
be the subspace of maps with the boundary marked points ordered by position.
If in addition \hbox{$\eta\!\in\!\cD_\om(\al)$} for some $\al\!\in\!\cC_{\om}(Y)$, define 
$$\M_{\eta;J}\equiv\M^\st_{k_\bu(\eta),L_\bu(\eta)}(\be_\bu(\eta);J),\quad
\M^+_{\eta;J}\equiv\M^\st_{k_\bu(\eta)+1,L_\bu(\eta)}(\be_\bu(\eta);J).$$

\begin{dfn}\label{bndch_dfn}
Let $R$, $(X,\om)$, $Y$, $\os$, and $\al\!\equiv\!(\be,K,L)$ be as in Theorem~\ref{WelOpen_thm}.
A \sf{bounding chain} on~$(\al,J)$ is 
a collection $(\fb_{\al'})_{\al'\in\cC_{\om;\al}(Y)}$ of bordered pseudocycles into~$Y$ such~that
\BEnum{(BC\arabic*)}

\item\label{BCdim_it} $\dim\,\fb_{\al'}\!=\!\dim(\al')\!+\!2$ 
for all $\al'\!\in\!\cC_{\om;\al}(Y)$;

\item\label{BC0_it} $\fb_{\al'}\!=\!\eset$ unless
$\al'\!=\!(0,\{\pt\},\eset)$ for some $\pt\!\in\!K$ or 
$\dim(\al')\!=\!0$;

\item\label{BCi_it} $\fb_{(0,\{\pt\},\eset)}\!=\!\pt$ for all $\pt\!\in\!K$;

\item\label{BCprt_it} for all $\al'\!\in\!\cC_{\om;\al}(Y)$ such that 
$\dim(\al')\!=\!0$,
\BE{BCprt_e}\prt\fb_{\al'}=
\bigg(\!\evb_1\!:\!\!\!\bigcup_{\eta\in\cD_{\om}(\al')}\hspace{-.12in}
(-1)^{k_\bu(\eta)}\M^+_{\eta;J}\!\fiber\!\!
\big(\!(i\!+\!1,\fb_{\al_{i}(\eta)})_{i\in[k_{\bu}(\eta)]}; (i,\Ga_i)_{\Ga_i\in L_{\bu}(\eta)}\big)
\lra Y\!\bigg).\EE
\EEnum
\end{dfn}

Since the dimension of every pseudocycle $\Ga_i\!\in\!L$ is even,
the oriented morphism
\BE{fbbetadfn_e}
\fbb_{\eta}\!\equiv
 \bigg(\!\evb_1\!:(-1)^{k_\bu(\eta)}\M^+_{\eta;J}\!\fiber\!
\big(\!(i\!+\!1,\fb_{\al_{i}(\eta)})_{i\in[k_{\bu}(\eta)]};(i,\Ga_i)_{\Ga_i\in L_{\bu}(\eta)}\big)
\lra Y\!\bigg)\EE
in~\eref{BCprt_e} does not depend on the choice of identification of~$L_{\bu}(\eta)$
with~$[|L_{\bu}(\eta)|]$.
By Lemma~3.1 in~\cite{JakeSaraGeom}, the~map
\BE{fbbdfn_e}
\fbb_{\al'}\!\equiv\bigcup_{\eta\in\cD_{\om}(\al')}\hspace{-.1in}\!\!\fbb_{\eta}\EE
with orientation induced by the relative OSpin-structure $\os$ on~$Y$ is a pseudocycle 
for every \hbox{$\al'\!\in\!\cC_{\om;\al}(Y)\!\cup\!\{\al\}$}. 
If in addition $\dim(\al)\!=\!2$,
then $\fbb_{\al}$ is a pseudocycle of codimension~0.
Its degree determines a count of $J$-holomorphic disks in~$(X,Y)$
through $|K|\!+\!1$ points in~$Y$ as in~\eref{JSinvdfn_e2}.

A bounding chain $(\fb_{\al'})_{\al'\in\cC_{\om;\al}(Y)}$ as in Definition~\ref{bndch_dfn}
can also be used to define the counts~\eref{JSinvdfn_e} of $J$-holomorphic disks 
in the following way.
We denote the signed cardinality of a finite set~$S$ of signed points by~$|S|^{\pm}$. 
If $S$ is not a finite set of signed points, we set $|S|^{\pm}\!\equiv\!0$. 
If $\eta\!\in\!\cD_\om(\al)$, let
$$s^*(\eta)\equiv
\begin{cases}\frac{1}{k_\bu(\eta)}\!-\!\frac{1}{2},&
\hbox{if}~k_\bu(\eta)\!\neq\!0,\\
1,&\hbox{if}~k_\bu(\eta)\!=\!0.\end{cases}$$
Define
\BE{JSinvdfn_e2b}
\blr{L}_{\be;K}^{\om,\os}
\equiv \sum_{\eta\in\cD_{\om}(\al)}\hspace{-0.3cm}(-1)^{k_\bu(\eta)}s^*(\eta)
\Big|\M_{\eta;J}\!\fiber\!\!\big(\!(i,\fb_{\al_i(\eta)})_{i\in[k_\bu(\eta)]};
(i,\Ga_i)_{\Ga_i\in L_\bu(\eta)}\big)\Big|^{\pm}+\frac12\sum_{p\in K}
\!\blr{L}_{\!\be;K-\{p\}}^{\!\om,\os}.\EE
This number vanishes unless $\dim(\al)\!=\!0$.
Unlike~\eref{JSinvdfn_e2}, \eref{JSinvdfn_e2b} provides a definition 
of the counts~\eref{JSinvdfn_e} with $k\!=\!0$.
By Theorem~2.7(2) in~\cite{JakeSaraGeom}, 
$$\blr{L}_{\be;K}^{\om,\os}
~\hbox{in~\eref{JSinvdfn_e2b}} ~~=~~
\blr{L}_{\be;K-\{\pt\}}^{\om,\os}~\hbox{in~\eref{JSinvdfn_e2}}$$
for any $\pt\!\in\!K$ if
$\lr{L}_{\be;K-\{\pt\}}$ does not depend on~$\pt\!\in\!K$.

\section{Proof of Theorem~\ref{WelOpen_thm}}
\label{WelOpenPf_sec}

For the remainder of the paper, we take $(X,\om,Y)$ and~$\os$ as in Theorem~\ref{WelOpen_thm}.
Let $\ga_1,\ga_2$ be smooth maps from oriented closed one-manifolds 
into the oriented closed three-manifold~$Y$ with disjoint images.
If $\ga_1\!=\!\prt\fb_1$ and $\ga_2\!=\!\prt\fb_2$ for some bordered 
pseudocycles~$\fb_1$ and~$\fb_2$ into~$Y$
so that~$\fb_1$ is transverse to~$\ga_2$ and 
$\fb_2$ is transverse to~$\ga_1$, we define
\BE{lkdfn_e}\lk(\ga_1,\ga_2) \equiv \big|\fb_1\!\fiber\!\ga_2\big|^{\pm}
=-\big|\ga_1\!\fiber\!\fb_2\big|^{\pm}=\big|\fb_2\!\fiber\!\ga_1\big|^{\pm}
=-\big|\ga_2\!\fiber\!\fb_1\big|^{\pm}\,;\EE
the first and last equalities above hold by Lemma~\ref{fibersign_lmm},
while the middle one follows from Lemma~\ref{fibprodflip_lmm}.
The sign of a point $(p,q)$  of $\fb_1\!\fiber\!\ga_2$ is the sign of the isomorphism
$$T_p\dom(\fb_1)\!\oplus\!T_q\dom(\ga_2)\lra T_{\fb_1(p)}Y\!=\!T_{\ga_2(q)}Y, \qquad
(v,w)\lra \nd_p\fb_1(v)\!+\!\nd_q\ga_2(w).$$
The \sf{linking number}~\eref{lkdfn_e} of the one-cycles~$\ga_1$ and~$\ga_2$
that bound in~$Y$ does not depend on the choice of~$\fb_1,\fb_2$.
In this section, we take linking numbers of the boundaries~$\prt u$ of
$J$-holomorphic maps~$u$ from $(\D^2,S^1)$ to~$(X,Y)$.
By the injectivity of~\eref{iohomdfn_e}, these boundaries also bound in~$Y$
and thus have well-defined linking numbers.

\subsection{Bounding chains and Welschinger's invariants}
\label{BCtoWel_subs}

For $\be_1,\ldots,\be_m\!\in\!H_2^{\om}(X,Y)$, 
we denote by $\M^\uo_{K,L}(\be_1,\ldots,\be_m;J)$ 
the moduli space of unions of~$m$ $J$-holomorphic disks 
in classes $\be_1,\ldots,\be_m$ with
$L$-labeled interior marked points and $K$-labeled boundary marked points
between the $m$~disks.
In contrast to Section~2.4 in~\cite{Wel13}, we do {\it not} order the disks or
orient this moduli space. 
Let
\BE{Mmultidfn_e}\M^{\uo,\circ}_{K,L}(\be_1,\ldots,\be_m;J)\subset \M^\uo_{K,L}(\be_1,\ldots,\be_m;J)\EE
be the dense open subset of the multi-disks whose $m$ components have pairwise disjoint 
boundaries in~$Y$. 
We extend the definitions of the evaluations maps~$\evb_i$ and~$\evi_i$
and of the associated fiber product~$\fiber$ of Section~\ref{Ms_subs}
to the moduli spaces in~\eref{Mmultidfn_e}.

Let $\al\!\equiv\!(\be,K,L)$ and $J$ be as in Theorem~\ref{WelOpen_thm}
and $p_1,\ldots,p_k$ be an ordering of the elements of~$K$.
For any element $\al'\!\equiv\!(\be',K',L')$ of $\cC_{\om;\al}(Y)\!\cup\!\{\al\}$, 
we endow $K'\!\subset\!K$ with the order induced from~$K$.
We define the spaces of (constrained) \sf{single $\al'$-disks} and \sf{$\al'$-multi-disks} by
\begin{equation*}\begin{split}
\SDC(\al')&\equiv\M^\uo_{K',L'}(\be';J)\!\!\fiber\!\!
\big(\!(i,p_i)_{i\in K'};(i,\Ga_i)_{\Ga_i\in L'}\big)
\qquad\hbox{and}\\
\MDC(\al')&\equiv \bigsqcup_{m=1}^{\i}
\bigsqcup_{\begin{subarray}{c}
\be_1,\ldots,\be_m\in H_2^{\om}(X,Y)\\ \tn{unordered}\\ \be_1+\ldots+\be_m=\be'\end{subarray}} 
\hspace{-.4in}
\M^{\uo,\circ}_{K',L'}(\be_1,\ldots,\be_m;J)\!\!\fiber\!\!
\big(\!(i,p_i)_{i\in K'};(i,\Ga_i)_{\Ga_i\in L'}\big),
\end{split}\end{equation*}
respectively.
We write an element $\u$ of $\MDC(\al')$ as
\begin{gather}\label{uMDCdfn_e} 
\u\equiv\big[\u_1,\ldots,\u_m\big] \quad\hbox{with}\quad
\u_r\in \M^\uo_{K_r,L_r}(\be_r;J)\!\!
\fiber\!\!\big(\!(i,p_i)_{i\in K_r};(i,\Ga_i)_{\Ga_i\in L_r}\big)\\
\notag 
\hbox{for some}\quad m,\be_r,K_r,L_r~~\hbox{with}~~
\be_1\!+\!\ldots\!+\!\be_m=\be',~~\bigsqcup_{r=1}^m\!K_r=K',~~\bigsqcup_{r=1}^m\!L_r=L'\,.
\end{gather}
For such an element $\u$ of $\MDC(\al')$, we write 
$\u_r\!\in\!\u$ to indicate that~$\u_r$ is a component of the multi-disk~$\u$.
Let 
$$\prt\u: \unbr{S^1\!\sqcup\!\ldots\!\sqcup\!S^1}{m} \lra Y$$
be the boundary of the components of $\u$ with the orientation induced by 
the complex orientation on the unit disk.
If $\dim(\al')\!=\!0$ and $\u_r\!\in\!\u$, 
we denote by $\sgn(\u_r)$ the sign of $\u_r$
as an element of the fiber product in~\eref{uMDCdfn_e} and set
$$\sgn(\u)\equiv\prod_{\u_r\in\u}\!\!\sgn(\u_r)\,;$$
this sign does not depend on the order on~$K$.
If $\dim(\al')\!\neq\!0$, we define $\sgn(\u)\!\equiv\!0$.

For $\u\!\in\!\MDC(\al')$ as in~\eref{uMDCdfn_e}, 
we denote by~$K_\u$ the complete graph with vertices~$\u_1,\ldots,\u_m$.
We call a tree $T\!\subset\!K_{\u}$, i.e.~a connected subgraph without loops,
\sf{spanning} if~$T$ contains all vertices of~$K_{\u}$
and denote by $\ST(\u)$ the set of all spanning trees $T\!\subset\!K_{\u}$.
Let 
$$\lk(\u;T)\equiv
\prod_{\begin{subarray}{c}\tn{edge}\,e\in T\\
\tn{connecting}\,\u_r,\u_s\end{subarray}}
\hspace{-.35in}\lk\big(\prt\u_r,\prt\u_s\big)~~\forall\,T\!\in\!\ST(\u)
\quad\hbox{and}\quad 
\lk(\u)\equiv\sum_{T\in\ST(\u)}\!\!\!\!\!\!\lk(\u;T)\,.$$

\vspace{-.1in}

Welschinger's definition of the open GW-invariant~\eref{Welinvdfn_e} in 
\cite[Sect.~4.1]{Wel13} is equivalent~to
\BE{Welinvdfn_e0}
\bllrr{\Ga_1,\ldots,\Ga_l}_{\be,|K|}^{\om,\os}
=\sum_{\u\in\MDC(\al)}\hspace{-.15in}\sgn(\u)\lk(\u)\,.\EE
The first statement of the next proposition establishes~\ref{WelBC_it}.
We deduce~\ref{Welgen_it} from the second statement of this proposition
and Proposition~\ref{RDivRel_prp}.
The two propositions are proved in Sections~\ref{MainPf_subs} and~\ref{OpenDivRel_subs}.

\begin{prp}\label{WelOpen_prp}
Let $\al\!\in\!\cC_{\om}(Y)$ and $J$ be as in Theorem~\ref{WelOpen_thm}.
\BEnum{(\arabic*)}

\item  There exists a bounding chain $(\fb_{\al'})_{\al'\in\cC_{\om;\al}(Y)}$
on~$(\al,J)$.

\item For every such bounding chain~$(\fb_{\al'})_{\al'\in\cC_{\om;\al}(Y)}$
and $\al'\!\in\!\cC_{\om;\al}(Y)$ with $\dim(\al')\!=\!0$,
the associated closed pseudocycle~\eref{fbbdfn_e} satisfies
\BE{fbbalWel_e}
\prt\fb_{\al'}=\fbb_{\al'}
=(-1)^{|K(\al')|}\!\!\!\!\!\!\!
\bigsqcup_{\u\in\MDC(\al')}\!\!\!\!\!\!\sgn(\u)\lk(\u)\prt\u\,.\EE
\EEnum
\end{prp}

\vspace{.1in}

If $k\!\in\!\Z^+$ and $K\!\subset\![k]$, let
$$\ff^b_k\!:\M^{\uo,\st}_{k,l}(\be;J)\lra\M^{\uo,\st}_{k-1,l}(\be;J) \quad\hbox{and}\quad 
\ff^b_{k;K}\!:
\M^{\uo,\st}_{k,l}(\be;J)\lra\M^{\uo,\st}_{[k]-K,l}(\be;J)$$
be the forgetful morphisms dropping the boundary marked point with index~$k$
and the boundary marked points indexed by the set~$K$,
respectively.

\begin{prp}[Open Divisor Relation]\label{RDivRel_prp}
Suppose $K'\!\subset\!K\!\subset\![k]$, $L\!\subset\![l]$, and 
$(\fb_i)_{i\in K}$ and $(\Ga_i)_{i\in L}$ are
tuples of bordered pseudocycles into~$Y$ and~$X$, respectively.
If the codimension of~$\fb_i$ is~1 for every $i\!\in\!K'$ and
$$K'\subset\big\{k',\ldots,k\big\}\subset K$$
for $k'\!\in\![k]$,
then there exists a dense open subset $\M_{[k]-K',l}^*$ of 
the target of the induced forgetful morphism
\BE{fk2_e}\begin{split}
&\ff^b_{k;K'}\!:
\M^{\uo,\st}_{k,l}(\be;J)\!\!\fiber\!\!\big(\!(i,\fb_i)_{i\in K};(i,\Ga_i)_{i\in L}\big)\\
&\hspace{2in}
\lra\M^{\uo,\st}_{[k]-K',l}(\be;J)\!\!\fiber\!\!
\big(\!(i,\fb_i)_{i\in K-K'};(i,\Ga_i)_{i\in L}\big)
\end{split}\EE
so that \eref{fk2_e} restricts to a covering map over each connected component~$\M$
of~$\M_{[k]-K',l}^*$.
If in addition the codimensions of all~$\fb_i$ are odd and the codimensions of~$\Ga_i$ are even,
then 
$$\deg\big(\ff^b_{k;K'}\big|_{\M}\big)=(-1)^{|K'|}
\!\prod_{i\in K'}\!\lk(\prt\u,\prt\fb_i)$$ 
for any $\u\!\in\!\M$. 
\end{prp}

\begin{rmk}\label{WelReal_rmk}
Suppose that $\phi$ is an anti-symplectic involution on~$(X,\om)$,
i.e.~$\phi^2\!=\!\id_X$ and $\phi^*\om\!=\!-\om$, 
$Y\!\subset\!X$ is a topological component of
the fixed locus~$X^{\phi}$ of~$\phi$, $\phi^*J\!=\!-J$, and for every $i\!\in\![l]$
there exists a diffeomorphism~$\phi_{\Ga_i}$ of~$\dom\,\Ga_i$ such~that 
$$\phi\!\circ\!\Ga_i=\Ga_i\!\circ\!\phi_{\Ga_i} \qquad\hbox{and}\qquad
\sgn\,\phi_{\Ga_i}=-(-1)^{(\dim\,\Ga_i)/2}\,.$$
Let $H_{2;\phi}(X,Y)$ be the quotient of $H_2(X,Y;\Z)$ by the image of 
the endomorphism $\{\Id\!+\!\phi_*\}$.
For $B\!\in\!H_{2;\phi}(X,Y)$, let 
$$\SDC(B,K,L)= 
\bigsqcup_{\begin{subarray}{c}\al'\in\cC_{\om}(Y), \be(\al')\in B\\
K(\al')=K,L(\al')=L\end{subarray}}\hspace{-.35in}\SDC(\al')\,, \qquad
\MDC(B,K,L)= 
\bigsqcup_{\begin{subarray}{c}\al'\in\cC_{\om}(Y), \be(\al')\in B\\
K(\al')=K,L(\al')=L\end{subarray}}\hspace{-.35in}\MDC(\al')\,.$$
If $\fc$ denotes the complex conjugation on~$\D^2$,
the replacement of $\u_r\!\in\!\u$ as in~\eref{uMDCdfn_e} with
\begin{equation*}\begin{split}
&\u_r'\!\equiv\!
\big(\big[\phi\!\circ\!u_r\!\circ\!\fc,(x_i)_{i\in K_r},(\fc(z_i)\!)_{\Ga_i\in L_r}\big],
(i,p_i)_{i\in K_r};(i,\phi_{\Ga_i}(q_i)\!)_{\Ga_i\in L_r}\big)\\
&\hspace{2in}\in \M^\uo_{K_r,L_r}(-\phi_*(\be_r);J)\!\!
\fiber\!\!\big(\!(i,p_i)_{i\in K_r};(i,\Ga_i)_{\Ga_i\in L_r}\big)
\end{split}\end{equation*}
preserves $\MDC(B,K,L)$.
If $\u'\!\in\!\MDC(B,K,L)$ is the resulting element, $\sgn(\u')\!=\!\sgn(\u)$
by \cite[Prop.~5.1]{Jake}; see also \cite[Lem.~B.7]{JakeSaraGeom}.
However, $\prt\u_r$ and $\prt\u_r'$ are the same circles with the opposite orientations.
If precisely one edge of a tree $T\!\in\!\ST(\u)$ contains $\u_r\!\in\!\u$ as a vertex,
this implies that $\lk(\u;T)\!=\!-\lk(\u';T')$, where $T'$ is the tree obtained from~$T$
by replacing the vertex~$\u_r$ with~$\u_r'$.
It follows that the collection
$$\big\{(\u,T)\!: \u\!\in\!\MDC(B,K,L)\!-\!\SDC(B,K,L),\,T\!\in\!\ST(\u)\big\}$$
can be split into pairs of elements with opposite values of $\sgn(\u)\lk(\u;T)$. 
Thus,
$$\sum_{\be\in B}\!\!
\bllrr{\Ga_1,\ldots,\Ga_l}_{\be,|K|}^{\om,\os}
\equiv\sum_{\u\in\MDC(B,K,L)}\sum_{T\in\ST(\u)}\!\!\!\!\!\!\sgn(\u)\lk(\u;T)
=\sum_{\u\in\SDC(B,K,L)}\hspace{-.25in}\sgn(\u)\,.$$
The right-hand side above is the GW-invariant of $(X,\om,\phi)$ 
enumerating degree~$B$ rational $J$-holomorphic curves through the specified 
constraints as defined in~\cite{RealWDVV3}.
This invariant is a re-interpretation of the invariants defined
in~\cite{Wel6,Wel6b,Jake}; see also \cite[(B.12)]{JakeSaraGeom}. 
\end{rmk}

\subsection{Main argument}
\label{MainPf_subs}

We continue with the setting of Theorem~\ref{WelOpen_thm} and Proposition~\ref{WelOpen_prp}.
Let \hbox{$\al'\!\in\!\cC_{\om;\al}(Y)\!\cup\!\{\al\}$}.
For $\eta\!\in\!\cD_\om(\al')$, define 
\begin{gather*}
K^*_\bu(\eta)\equiv\big\{i\!\in\![k_\bu(\eta)]\!:
\al_i(\eta)\!\neq\!(0,\{\pt\},\eset)~\forall\,\pt\!\in\!K\big\}, \\
K^\pt_\bu(\eta)\equiv\big\{\pt\!\in\!K\!:
(0,\{\pt\},\eset)\!=\!\al_i(\eta)~\tn{for some}~i\!\in\![k_{\bu}(\eta)]\big\},\\
\al_{\bu}^{\pt}(\eta)\equiv\big(\be_\bu(\eta),K^\pt_\bu(\eta),L_\bu(\eta)\!\big)
\in\cC_{\om;\al}(Y)\!\cup\!\{\al\}, \quad
\M^{\uo,+}_{\eta;J}=\M^{\uo,\st}_{[k_\bu(\eta)+1],L_\bu(\eta)}\big(\be_\bu(\eta);J\big).
\end{gather*}
For $\eta,\eta'\!\in\!\cD_\om(\al')$, define $\eta\!\sim\!\eta'$ if 
\begin{gather*}
\big(\be_{\bu}(\eta),k_\bu(\eta),L_\bu(\eta)\!\big)=
\big(\be_{\bu}(\eta'),k_\bu(\eta'),L_\bu(\eta')\big)
\qquad\hbox{and}\\
\big(\al_i(\eta)\!\big)_{i\in[k_\bu(\eta)]}\hbox{ is a permutation of }
\big(\al_i(\eta')\!\big)_{i\in[k_\bu(\eta)]}.
\end{gather*}
Denote by $[\eta]$ the equivalence class of~$\eta$.
With $\fbb_{\eta}$ as in~\eref{fbbetadfn_e}, let
$$\fbb_{[\eta]}=\bigsqcup_{\eta'\in[\eta]}\!\!\!\fbb_{\eta'}\,.$$ 

\vspace{-.1in}

We define
\begin{gather}\notag
\DMDC(\al')\equiv\big\{(\u,\u_{\bu},T)\!:\u\!\in\!\MDC(\al'),\,\u_{\bu}\!\in\!\u,\,
T\!\in\!\ST(\u)\big\},\\ 
\label{ovDMDCdfn_e0}
\wt\DMDC(\al')\equiv
\Big\{\!
\big(\eta,\u_{\bu},(\wt\u_i)_{i\in K_{\bu}^*(\eta)}\big)\!:
\eta\!\in\!\cD_{\om}(\al'),\,\u_{\bu}\!\in\!\SDC(\al_{\bu}^{\pt}(\eta)\!),
\wt\u_i\!\in\!\DMDC(\al_i(\eta)\!)~
\forall\,i\!\in\!K_{\bu}^*(\eta)\Big\}.
\end{gather}
We define elements $(\eta,\u_{\bu},(\wt\u_i)_{i\in K_{\bu}^*(\eta)})$ and 
$(\eta',\u_{\bu}',(\wt\u_i')_{i\in K_{\bu}^*(\eta')})$ of the last space to be
\sf{equivalent}
if $\u_{\bu}\!=\!\u_{\bu}'$, $k_{\bu}(\eta)\!=\!k_{\bu}(\eta')$,
and there exists a permutation~$\si$ of $[k_{\bu}(\eta)]$
such~that 
$$\al_i(\eta)=\al_{\si(i)}(\eta')~~\forall\,i\!\in\!\big[k_{\bu}(\eta)\big], \quad
\si\big(K_{\bu}^*(\eta)\!\big)\subset K_{\bu}^*(\eta'),
\quad\hbox{and}\quad \wt\u_i=\wt\u_{\si(i)}'~~\forall\,i\!\in\!K_{\bu}^*(\eta).$$
We denote by $\ov{\DMDC}(\al')$ the quotient of 
the space in~\eref{ovDMDCdfn_e0} by this equivalence relation.

Let $(\u,\u_{\bu},T)\!\in\!\DMDC(\al')$. 
For each $\u_r\!\in\!\u$, let 
$$\be(\u_r)\in H_2^{\om}(X,Y), \qquad L(\u_r)\!\subset\!L(\al'),
\quad\hbox{and}\quad K(\u_r)\!\subset\!K(\al')$$
be the degree of $\u_r$, the interior marked points carried by~$\u_r$,
and the boundary marked points carried by~$\u_r$, respectively.
We denote by $\Br(\u_{\bu};T)$ the set of branches of~$T$ at~$\u_{\bu}$,
i.e.~the trees~$T_i$ obtained by removing the vertex~$\u_{\bu}$ from the graph~$T$.
For each $i\!\in\!\Br(u_{\bu};T)$, we denote by~$\u_i'$ the set of all vertices of~$T_i$
and by $\u_{i\bu}'\!\in\!\u_i'$ the vertex connected by an edge of~$T$ to~$\u_{\bu}$. 
Define
$$\be_i=\sum_{\u_r\in\u_i'}\!\!\!\be(\u_r), \quad K_i=\bigsqcup_{\u_r\in\u_i'}\!\!\!K(\u_r),
\quad L_i=\bigsqcup_{\u_r\in\u_i'}\!\!\!L(\u_r),
\quad \al_i=(\be_i,K_i,L_i)\in\cC_{\om;\al}(Y).$$
Let $\al_{\pt}\!=\!(0,\{\pt\},\eset)$ for each $\pt\!\in\!K(\u_{\bu})$ and
$$k_{\bu}=\big|K(\u_{\bu})\big|\!+\!\big|\Br(\u_{\bu};T)\big|.$$
Identifying $K(\u_{\bu})\!\sqcup\!\Br(\u_{\bu};T)$ with $[k_{\bu}]$,
we obtain an element
$$\big( \eta\!\equiv\!\big(\be(\u_{\bu}),k_{\bu},L(\u_{\bu}),(\al_i)_{i\in[k_{\bu}]}\big),
\u_{\bu},(\u_i',\u_{i\bu}',T_i)_{i\in\Br(\u_{\bu};T)}\big)\in \wt\DMDC(\al').$$
The induced element of $\ov{\DMDC}(\al')$ does not depend on the choice of this identification.
In this way, we obtain a natural bijection
\BE{DMDCmap_e} \DMDC(\al')\lra\ov\DMDC(\al').\EE

\vspace{-.1in}

For $k\!\in\!\Z^{\ge0}$, we denote by $\bS_k$ the $k$-th symmetric group.
For $\si\!\in\!\bS_{k_\bu(\eta)}$, define 
\begin{equation*}\begin{split}
\io_{\eta;\si}^+\!:\M^+_{\eta;J}&\lra\M^{\uo,+}_{\eta;J},\\
\io_{\eta;\si}^+\big(u;x_1,x_2,\ldots,x_{k_\bu(\eta)+1},(z_i)_{\Ga_i\in L_\bu(\eta)}\big)
&=\big(u;x_1,x_{1+\si(1)},\ldots,x_{1+\si(k_\bu(\eta))},(z_i)_{\Ga_i\in L_\bu(\eta)}\big).
\end{split}\end{equation*}
This map is an open embedding and 
$$\M^{\uo,+}_{\eta;J}
=\bigsqcup_{\si\in\bS_{k_\bu(\eta)}}\!\!\!\!\!(\sgn\,\si)\big(\Im\,\io_{\eta;\si}^+\big).$$
If $\eta\!\sim\!\eta'$ are such that $\al_i(\eta)=\al_{\si(i)}(\eta')$ 
for all $i\in[k_\bu(\eta)]$,  then 
$$(-1)^{k_\bu(\eta)}\fbb_{\eta'}
\approx\Big(\!\evb_1\!:(\sgn\,\si)\big(\Im\,\io_{\eta;\si}^+\big)
\!\!\fiber\!\!\big(\!(i\!+\!1,\fb_{\al_i(\eta)})_{i\in[k_\bu(\eta)]};
(i,\Ga_i)_{\Ga_i\in L_\bu(\eta)}\big)
\lra Y\!\Big)$$
by Lemma~\ref{fibprodflip_lmm}.
Therefore, 
\BE{WelOpen_e9}
\fbb_{[\eta]}\approx(-1)^{k_\bu(\eta)}\!\Big(\!\evb_1\!:
\M^{\uo,+}_{\eta;J}\!\!\fiber\!\!\big(\!(i\!+\!1,\fb_{\al_i(\eta)})_{i\in[k_\bu(\eta)]};
(i,\Ga_i)_{\Ga_i\in L_\bu(\eta)}\big)\lra Y\!\Big).\EE

\begin{proof}[{\bf{\emph{Proof of~\ref{Welgen_it}}}}]
We establish this statement with $K$ replaced by~$K\!-\!\{p_1\}$ under 
the assumption that $\dim(\al)\!=\!0$.

Let $\al'$ and $\eta$ be as above with $1\!\not\in\!K(\al')$.
With $\fp_1\!\equiv\!(0,\{p_1\},\eset)$, 
$$\fbb_{[\eta]}\!\fiber\!\fb_{\fp_1}
=\unbr{(-1)^{k_\bu(\eta)}}{{\tn{~\eref{WelOpen_e9}}}}
\unbr{(-1)^{k_\bu(\eta)}}{{\tn{~Lemma~\ref{fibprodflip_lmm}}}} 
\unbr{(-1)^{k_\bu(\eta)}}{\tn{~Lemma~\ref{fibprodisom_lmm1}}}
\M^{\uo,+}_{\eta;J}\!\!\fiber\!\!
\big(\!(1,\fb_{\fp_1}),(i\!+\!1,\fb_{\al_i(\eta)})_{i\in[k_\bu(\eta)]};
(i,\Ga_i)_{\Ga_i\in L_\bu(\eta)}\big).$$
Suppose in addition $\dim(\al')\!=\!2$. 
By the above identity, Proposition~\ref{RDivRel_prp} with \hbox{$K'\!=\!K_{\bu}^*(\eta)$}, 
and~\eref{fbbalWel_e}
with~$\al'$ replaced by~$\al_i(\eta)$,
\BE{WelOpen_e10}\begin{split}
\big|\fbb_{[\eta]}\!\fiber\!\fb_{\fp_1}\big|^{\pm}
&=(-1)^{|K^\pt_\bu(\eta)|}\hspace{-.35in}
\sum_{\u_{\bu}\in\SDC(\al_{\bu}^{\pt}(\eta)+\fp_1)}\hspace{-.25in}
\Big(\sgn(\u_{\bu})\hspace{-.15in}\prod_{i\in K^*_\bu(\eta)}\hspace{-.15in}
\lk(\prt\u_{\bu},\prt\fb_{\al_i(\eta)})\!\!\Big)\\
&=(-1)^{|K(\al')|}\hspace{-.35in}
\sum_{\u_{\bu}\in\SDC(\al_{\bu}^{\pt}(\eta)+\fp_1)}\hspace{-0.1in}
\bigg(\!\sgn(\u_{\bu})\hspace{-.15in}\prod_{i\in K^*_\bu(\eta)}
\hspace{-.1in}\Big(
\sum_{\begin{subarray}{c}\u_i\in\MDC(\al_i(\eta)\!)\\ 
T_i\in\ST(\u_i)\end{subarray}}\hspace{-.28in}
\sgn(\u_i)\lk(\u_i;T_i)\lk(\prt\u_{\bu},\prt\u_i)\!\!\Big)\!\!\!\bigg)\,.
\end{split}\EE
Since the dimension of~$Y$ is odd, 
$$-\deg\fbb_{\al'}=\big|\fbb_{\al'}\!\fiber\!\fb_{\fp_1}\big|^{\pm}
\equiv\sum_{[\eta]\in\cD_{\om}(\al')/\sim}
\hspace{-.25in}\big|\fbb_{[\eta]}\!\fiber\!\fb_{\fp_1}\big|^{\pm}
\,.$$
Summing up~\eref{WelOpen_e10} over the equivalence classes~$[\eta]$ of~$\eta$
in~$\cD_{\om}(\al')$
and using the bijectivity of the map~\eref{DMDCmap_e}, we thus obtain
$$-\deg\fbb_{\al'}=(-1)^{|K(\al')|}\hspace{-.3in}
\sum_{\begin{subarray}{c}\u\in\MDC(\al'+\fp_1)\\ 
\u_{\bu}\in\u,T\in\ST(\u)\\
\u_{\bu}~\tn{passes thr.}~p_1
\end{subarray}}\hspace{-.35in}\sgn(\u)\lk(\u;T)
=(-1)^{|K(\al')|}\hspace{-.25in}
\sum_{\u\in\MDC(\al'+\fp_1)}\hspace{-.25in}\sgn(\u)\lk(\u)\,.$$
Taking $\al'\!=\!(\be,K\!-\!\{p_1\},L)$ above and using~\eref{JSinvdfn_e2} 
and~\eref{Welinvdfn_e0}, we obtain
$$\blr{L}_{\be,K-\{p_1\}}^{\om,\os}
=(-1)^{|K|}\bllrr{\Ga_1,\ldots,\Ga_l}_{\be,|K|}^{\om,\os}$$ 
and establish the claim.
\end{proof}

\begin{proof}[{\bf{\emph{Proof of Proposition~\ref{WelOpen_prp}}}}]
We prove both statements by induction on the set~$\cC_{\om}(Y)$
with respect to the partial order~$\prec$ defined in Section~\ref{notation_subs}.
It is sufficient to consider 
the elements $\al'\!\in\!\cC_{\om}(Y)$ with $\dim(\al')\!=\!0$
only.

Suppose $\al\!\in\!\cC_{\om}(Y)$ with $\dim(\al)\!=\!0$
and $(\fb_{\al'})_{\al'\in\cC_{\om;\al}(Y)}$
is a collection of bordered pseudocycles into~$Y$ satisfying the conditions
of Definition~\ref{bndch_dfn} as well as the second equality in~\eref{fbbalWel_e}  
if $\dim(\al')\!=\!0$.
By~\eref{WelOpen_e9}, Proposition~\ref{RDivRel_prp}  with \hbox{$K'\!=\!K_{\bu}^*(\eta)$}, 
and~\eref{fbbalWel_e}
with~$\al'$ replaced by~$\al_i(\eta)$,
\BE{WelOpen_e11}\begin{split}
\fbb_{[\eta]}
&=(-1)^{|K^\pt_\bu(\eta)|}\hspace{-.25in}
\bigsqcup_{\u_{\bu}\in\SDC(\al_{\bu}^{\pt}(\eta))}\hspace{-.25in}
\Big(\sgn(\u_{\bu})\hspace{-.15in}\prod_{i\in K^*_\bu(\eta)}\hspace{-.15in}
\lk(\prt\u_{\bu},\prt\fb_{\al_i(\eta)})\!\!\Big)\prt\u_{\bu}\\
&=(-1)^{|K(\al)|}\hspace{-.25in}
\bigsqcup_{\u_{\bu}\in\SDC(\al_{\bu}^{\pt}(\eta))}\hspace{-0.1in}
\bigg(\!\sgn(\u_{\bu})\hspace{-.15in}\prod_{i\in K^*_\bu(\eta)}
\hspace{-.1in}\Big(
\sum_{\begin{subarray}{c}\u_i\in\MDC(\al_i(\eta)\!)\\ 
T_i\in\ST(\u_i)\end{subarray}}\hspace{-.28in}
\sgn(\u_i)\lk(\u_i;T_i)\lk(\prt\u_{\bu},\prt\u_i)\!\!\Big)\!\!\!\bigg)\prt\u_{\bu}\,.
\end{split}\EE
Summing up~\eref{WelOpen_e11} over the equivalence classes~$[\eta]$ 
of~$\eta$ in~$\cD_{\om}(\al)$
and using the bijectivity of the map~\eref{DMDCmap_e}, we obtain
\begin{equation*}\begin{split}
\fbb_{\al}\equiv \bigsqcup_{[\eta]\in\cD_{\om}(\al)/\sim}
\hspace{-.25in}\fbb_{[\eta]}
=(-1)^{|K(\al)|}\hspace{-.3in}
\bigsqcup_{\begin{subarray}{c}\u\in\MDC(\al)\\ 
\u_{\bu}\in\u,T\in\ST(\u)\end{subarray}}\hspace{-0.3in}\sgn(\u)\lk(\u;T)\prt\u_{\bu}
=(-1)^{|K(\al)|}\!\!\!\!\!\!\!
\bigsqcup_{\u\in\MDC(\al)}\!\!\!\!\!\!\sgn(\u)\lk(\u)\prt\u\,.
\end{split}\end{equation*}
This establishes the second equality in~\eref{fbbalWel_e} with~$\al'$ replaced by~$\al$.
Along with the injectivity of~\eref{iohomdfn_e}, it implies that~$\fbb_{\al}$
bounds in~$Y$.
Thus, we can choose a bordered pseudocycle~$\fb_{\al}$ into~$Y$
satisfying the first equality in~\eref{fbbalWel_e} with~$\al'$ replaced by~$\al$.
\end{proof}

\subsection{Open divisor relation}
\label{OpenDivRel_subs}

We deduce Proposition~\ref{RDivRel_prp} from the following lemma, 
which confirms the $K'\!=\!\{k\}$ case of this proposition.

\begin{lmm}\label{RDivRel_lmm}
Let $\be,k,l$, $K,L$, $(\fb_i)_{i\in K}$, and
$(\Ga_i)_{i\in L}$ be as in Proposition~\ref{RDivRel_prp}.
If $k\!\in\!K$ and the codimension of~$\fb_k$ is~1, 
then there exists a dense open subset~$\M_{k-1,l}^*$ of 
the target of the induced forgetful morphism
\BE{fk_e}
\ff^b_k\!:
\M^{\uo,\st}_{k,l}(\be;J)\!\!\fiber\!\!\big(\!(i,\fb_i)_{i\in K};(i,\Ga_i)_{i\in L}\big)
\lra\M^{\uo,\st}_{k-1,l}(\be;J)\!\!\fiber\!\!
\big(\!(i,\fb_i)_{i\in K-\{k\}};(i,\Ga_i)_{i\in L}\big)\EE
so that \eref{fk_e} restricts to a covering map over each connected component~$\M$
of~$\M_{k-1,l}^*$.
If in addition the codimensions of all~$\fb_i$ are odd and the codimensions of~$\Ga_i$ are even,
then the degree of this restriction is $-\lk(\prt\u,\prt\fb_k)$ for any $\u\!\in\!\M$. 
\end{lmm}

\begin{proof}
We denote the right-hand side of~\eref{fk_e} by~$M$ and define
$$K'=K-\{k\}, \qquad \wt{M}\equiv\M^{\uo,\st}_{k,l}(\be;J)\!\!\fiber\!\!
\big(\!(i,\fb_i)_{i\in K'};(i,\Ga_i)_{i\in L}\big).$$
By Lemma~\ref{fibprodisom_lmm1}, 
\BE{Welcomp2_e3}
\tn{LHS of \eref{fk_e}}=  
(-1)^{|K|-1}\big(\wt{M}_{\evb_k}\!\!\times_{\fb_k}\!\!(\dom\,\fb_k)\!\big). \EE
If the pseudocycles~$\fb_i$ with $i\!\in\!K$ have odd codimensions
and the pseudocycles~$\Ga_i$ have even codimensions, then
\BE{Welcomp2_e5}\dim\,\wt{M}_{\evb_k}\!\!\times_{\fb_k}\!\!(\dom\,\fb_k)
\cong k\!-\!|K| \qquad\hbox{mod}~2.\EE
Let $M'\!\subset\!M$ be the image  of the elements 
of the left-hand side of~\eref{fk_e} which meet the boundary 
of any of the pseudocycles~$\fb_i$ and~$\Ga_i$ or 
the pairwise intersection of any pair of these pseudocycles.
The dense open subset~$\M_{k-1,l}^*$ of~$M$ is 
the open subset of $M\!-\!M'$ consisting of the maps~$\u$ from~$\D^2$ 
with~$\prt\u$ transverse to~$\fb_k$.

\begin{figure}
\begin{small}
$$\xymatrix{& 0\ar[d]& 0\ar[d]& &\\
0\ar[r]& T_{x_k}S^1\ar[d]\ar[r]^{\Id}& T_{x_k}S^1\ar[d]\ar[r] & 0\ar[d] & \\
0\ar[r]& T_{\wt\u}\wt{M}\ar[d]\ar[r]&
T_{\wt\u}\big(\M^\uo_{k,l}(\be;J)\!\!\times\!\!\!\prod\limits_{i\in K'}\!\!(\dom\,\fb_i)
\!\!\times\!\!\!\prod\limits_{i\in L}\!(\dom\,\Ga_i)\!\big)
\ar[d]\ar[r]& 
\cN\De_{Y^{\!K'}\!\!\times\!X^{\!L}}\!\big|_y\ar[r]\ar[d]_{\Id}& 0\\
0\ar[r]& T_{\u}M\ar[d]\ar[r]& 
T_{\u}\big(\M^\uo_{k-1,l}(\be;J)\!\!\times\!\!\!\prod\limits_{i\in K'}\!\!(\dom\,\fb_i)
\!\!\times\!\!\!\prod\limits_{i\in L}\!(\dom\,\Ga_i)\!\big)\ar[d]\ar[r]& 
\cN\De_{Y^{\!K'}\!\!\times\!X^{\!L}}\!\big|_y\ar[r]\ar[d]&  0\\
& 0& 0&0 &}$$
$$\xymatrix{ && 0\ar[d]&& 0\ar[d]\\
& 0\ar[r]\ar[d] & T_{x_k}S^1\!\oplus\!T_{q_k}\!(\dom\,\fb_k)\ar[d]
\ar[rr]^>>>>>>>>>>>{\nd_{q_k}\!\fb_k-\nd_{x_k}\prt\u} && 
T_{\fb_k\!(q_k)}Y  \ar[d]_{\Id}\ar[r]& 0 & \\
0\ar[r]& T_{(\wt\u,q_k)}\!(\wt{M}_{\evb_k}\!\!\times_{\!\fb_k}\!\!(\dom\,\fb_k)\!) 
\ar[d]^{\Id}\ar[r]&
T_{\wt\u}\wt{M}\!\oplus\!T_{q_k}\!\dom(\fb_k) \ar[d]
\ar[rr]^>>>>>>>>>>>{\nd_{q_k}\!\fb_k-\nd_{\wt\u}\evb_k}&& 
T_{\fb_k\!(q_k)}Y \ar[r]\ar[d]& 0\\
0\ar[r] & T_{(\wt\u,q_k)}\!(\wt{M}_{\evb_k}\!\!\times_{\!\fb_k}\!\!(\dom\,\fb_k)\!)
\ar[r]^<<<<<<<<<<{\nd_{(\wt\u,q_k)}\ff^b_k}\ar[d]& T_\u M\ar[d]\ar[rr]&& 0\\
& 0 & 0&}$$
\end{small}
\caption{Commutative squares of exact sequences for the proof of Lemma~\ref{RDivRel_lmm}.}
\label{Welcomplmm2_fig}
\end{figure}

We compute the sign of~$\ff^b_k$ at a preimage~$(\wt\u,q_k)$  
of~$\u$ in the fiber product space in~\eref{Welcomp2_e3}
under~\eref{fk_e}.
Denote the $k$-th boundary marked point of~$\wt\u$ by~$x_k$
and the image of~$\u$ in $Y^{K'}\!\!\times\!X^I$ by~$y$.
All rows and the right column in the first diagram of Figure~\ref{Welcomplmm2_fig}
are orientation-compatible.
The short exact sequence 
$$0\lra T_{x_k}S^1\lra T_{\wt\u'}\M^\uo_{k,l}(\be;J) \lra
T_{\u'}\M^\uo_{k-1,l}(\be;J) \lra0,$$
where $\wt\u'$ and $\u'$ are the projections of $\wt\u$ and~$\u$, 
respectively, to the corresponding disk moduli spaces, has sign  $(-1)^{k-1}$.
Along with Lemma~6.3 in~\cite{RealWDVV}, this implies that 
the middle and left columns in the first diagram also have signs~$(-1)^{k-1}$. 
Thus, the middle column in the second diagram
has sign $(-1)^{k-1}$ as~well.
The middle row and the side columns in this diagram 
are orientation-compatible. 
The sign of the top row is the sign of $(x_k,q_k)$ in the fiber product $(\prt\u)\!\fiber\!\fb_k$.
Along with Lemma~6.3 in~\cite{RealWDVV} and~\eref{Welcomp2_e5}, this implies 
that the sign of the bottom row is~$(-1)^{|K|-1}$ times 
the sign of $(x_k,q_k)$ in the fiber product $(\prt\u)\!\fiber\!\fb_k$.

Combining the last conclusion with~\eref{Welcomp2_e3}, we obtain
$$\sum_{(\wt\u,q_k)\in\{\ff^b_k\}^{-1}(\u)}\hspace{-.32in}
\sgn\big(\nd_{(\wt\u,q_k)}\ff^b_k\big)
=\big|(\prt\u)\!\fiber\!\fb_k\big|^{\pm}\,.$$
Along with~\eref{lkdfn_e}, this establishes the degree claim.
\end{proof}

\begin{proof}[{\bf{\emph{Proof of Proposition~\ref{RDivRel_prp}}}}]
The first claim follows immediately from the first claim of Lemma~\ref{RDivRel_lmm}.
By Lemma~\ref{fibprodflip_lmm}, a reordering of the pseudocycles~$\fb_i$'s 
with $i\!=\!k',\ldots,k$
does not change the oriented space on the left-hand side of~\eref{fk2_e}.
We can thus assume~that 
$$K'=\big\{k\!-\!|K'|\!+\!1,\ldots,k\big\}\subset[k].$$
The second claim then follows from the second claim of Lemma~\ref{RDivRel_lmm} 
by induction.
\end{proof}

\vspace{.2in}

{\it Department of Mathematics, Stony Brook University, Stony Brook, NY 11794\\
xujia@math.stonybrook.edu}

\end{document}